\documentclass[runningheads]{llncs}
\usepackage{graphicx}
\usepackage{amsmath,amssymb,mathtools}
\usepackage{thmtools,thm-restate}

\allowdisplaybreaks
\newcommand{\R}{\mathbb{R}}
\newcommand{\Z}{\mathbb{Z}}
\newcommand{\N}{\mathbb{N}}

\DeclareMathOperator{\Var}{Var}

\DeclareMathOperator{\val}{val}
\DeclareMathOperator{\Prob}{\Pr}
\DeclareMathOperator{\Exp}{\mathbb{E}}
\DeclareMathOperator{\Cov}{Cov}
\newcommand{\E}{\Exp}
\newcommand{\T}{\mathsf{T}}
\newcommand{\finalHpUpperBoundConstant}{10^{15}}
\newcommand{\finalHpUpperBound}{\finalHpUpperBoundConstant  t\cdot \frac{m^{2.5}(\log n + m)^2}{n}}
\usepackage{cleveref}

\renewcommand{\phi}{\varphi}

\newcommand{\set}[1]{\{{#1}\}}

\newcommand{\floor}[1]{\lfloor{#1}\rfloor}
\newcommand{\ceil}[1]{\lceil{#1}\rceil}
\newcommand{\len}[1]{\lvert{#1}\rvert}
\newcommand{\lenfit}[1]{\left\lvert{#1}\right\rvert}
\newcommand{\length}[1]{\lVert{#1}\rVert}

\begin{document}

\title{On the Integrality Gap of Binary Integer Programs with Gaussian Data }

\author{Sander Borst$^\star$ \and
Daniel Dadush\thanks{This project has received funding from the European Research Council (ERC) under the European Union's Horizon 2020 research and innovation programme (grant agreement QIP--805241)} \and
Sophie Huiberts \and Samarth Tiwari$^\star$}
\authorrunning{Borst et al.}

\institute{Centrum Wiskunde \& Informatica (CWI), Amsterdam, The Netherlands
\\ \email{\{sander.borst,dadush,s.huiberts,samarth.tiwari\}@cwi.nl}}
\maketitle              % typeset the header of the contribution
\begin{abstract}
For a binary integer program (IP) $\max c^\T x, Ax \leq b, x \in \{0,1\}^n$,
where $A \in \R^{m \times n}$ and $c \in \R^n$ have independent Gaussian
entries and the right-hand side $b \in \R^m$ satisfies that its negative coordinates have $\ell_2$ norm at most $n/10$, we prove that the gap between the value of the linear
programming relaxation and the IP is upper bounded by
$\operatorname{poly}(m)(\log n)^2 / n$ with probability at least
$1-2/n^7-2^{-\operatorname{poly}(m)}$. Our results give a Gaussian analogue of the
classical integrality gap result of Dyer and Frieze (Math. of O.R., 1989) in
the case of random packing IPs. In constrast to the packing case, our
integrality gap depends only polynomially on $m$ instead of exponentially.
Building upon recent breakthrough work of Dey, Dubey and Molinaro (SODA,
2021), we show that the integrality gap implies that branch-and-bound requires $n^{\operatorname{poly}(m)}$ time on random Gaussian IPs with
good probability, which is polynomial when the number of constraints $m$ is
fixed. We derive this result via a novel meta-theorem, which relates the size
of branch-and-bound trees and the integrality gap for random
\emph{logconcave} IPs. 

\end{abstract}

\section{Introduction}

Consider the following linear program with $n$ variables and $m$ constraints
\begin{align*}
	\val_\mathsf{LP}(A,b,c) := \max_x         & \ \val(x)=c^\T x \\
        \text{s.t. } & Ax\leq b             \tag{Primal LP}\label{primal-lp}\\
                     & x \in [0,1]^n
\end{align*}
Let $\val_\mathsf{IP}(A,b,c)$ be the value of the same optimization problem
with the additional restriction that $x$ is integral, i.e., $x\in \{0,1\}^n$.
Now we define the integrality gap to be the quantity
$\mathsf{IPGAP}(A,b,c):=\val_\mathsf{LP}(A,b,c)-\val_\mathsf{IP}(A,b,c)$.    

The integrality gap of
integer linear programs forms an important measure for the complexity of
solving said problem in a number of works on the average-case complexity of
integer programming
\cite{beier_core_2004,dey_branch-and-bound_2021,dyer_gap_1992,dyer_probabilistic_1989,goldberg_finding_1984,lueker_average_1982}.

So far, probabilistic analyses of the integrality gap have focussed
on $0$--$1$ packing IPs and the generalized assignment problem.
In particular,
the entries of $A \in \R^{m \times n}, b \in \R^m, c \in \R^n$
in these problems are all non-negative, and the entries of $b$
were assumed to scale linearly with $n$.

In this paper, we analyze the integrality gap of \eqref{primal-lp}
under the assumption that the entries
of $A$ and $c$ are all independent Gaussian $\mathcal{N}(0,1)$ distributed,
and that the negative part of $b$ is small: $\|b^-\|_2 \leq n/10$.

We prove that, with high probability, the integrality gap $\mathsf{IPGAP}(A,b,c)$ is
small, i.e., \eqref{primal-lp} admits a solution $x \in \{0,1\}^n$ with value
close to the optimum.

\begin{theorem}\label{thm:ipgap_main_result}
There exists an absolute constant $C \geq 200$, such that, for $m \geq 1$, $n
\geq Cm^{4.5}$, $b \in \R^m$ with $\|b^-\|_2 \leq n/10$, if
$A$ and $c$ have i.i.d.~$\mathcal{N}(0,1)$ entries, then
\begin{align*}
    \Prob\left(\mathsf{IPGAP}(A,b,c) \geq \finalHpUpperBoundConstant \cdot t\cdot \frac{m^{2.5} (m + \log n)^2}{n}\right)\leq 4\cdot \left(1-\frac{1}{25}\right)^{t}+n^{-7},
\end{align*}
for all $1 \leq t \leq \frac{n}{Cm^{2.5}(m+\log n)^2}$.
\end{theorem}

In the previous probabilistic analyses by
\cite{dyer_gap_1992,dyer_probabilistic_1989,lueker_average_1982}, it is
assumed that $b = \beta n$ for fixed $\beta \in (0,1/2)^m$
and the entries of $(A, c)$ are independently distributed uniformly in the
interval $[0,1]$. Those works prove a similar bound as above, except that in
their results the dependence on $m$ is exponential instead of polynomial.
Namely, for $\beta_{\min} := \min_{i \in [n]} \beta_i$, they require $n \geq (1/\beta_{\min})^m \geq 2^{m}$ and the integrality gap scales like
$O(1/\beta_{\min})^m \log^2 n / n$. We note that the integrality gap in
\Cref{thm:ipgap_main_result} does not depend on the ``shape'' of $b$ (other
than requiring $\|b^-\|_2 \leq n/10$). We give a high-level overview of the proof
of~\Cref{thm:ipgap_main_result} in subsection~\ref{sec:ipgap-overview},
describing the similarities and differences with the analysis of Dyer and
Frieze~\cite{dyer_probabilistic_1989}

Building on breakthrough work of Dey, Dubey and
Molinaro~\cite{dey_branch-and-bound_2021}, we show that the integrality gap
above also implies that branch-and-bound applied to the above IP produces a
tree of size at most $n^{\operatorname{poly}(m)}$ with good
probability. For this purpose, we give a novel meta-theorem relating the
integrality gap and the complexity of branch-and-bound for random
\emph{logconcave} IPs. We detail this in the next subsection.

\subsection{Relating the Integrality Gap to Branch-and-Bound}
\label{sec:bnb}

In recent breakthrough work, Dey, Dubey and
Molinaro~\cite{dey_branch-and-bound_2021} provided a framework for deriving
upper bounds on the size of branch-and-bound trees for random IPs with small
integrality gaps. Their framework consists of two parts. In the first part,
one deterministically relates the size any branch-and-bound tree using
best-bound first node selection to the size of knapsack polytopes whose
weights are induced by reduced costs and whose capacity is equal to the
integrality gap. We recall that in the best-bound first rule, the next node
to be processed is always the node whose LP relaxation value is the largest.
This is formally encoded by the following theorem, which corresponds to a
slightly adapted version of~\cite[Corollary 2]{dey_branch-and-bound_2021}. 

\begin{theorem} \label{thm:bnb}
Consider a binary integer program of the form
\begin{align}
\max         &~ c^\T x \nonumber \\
 \text{s.t. } & Ax\leq b \tag{Primal IP}\label{primal-ip} \\
	      & x \in \{0,1\}^m. \nonumber
\end{align}
Then, the best bound first branch-and-bound algorithm produces a tree of size 

\begin{equation}
n^{O(m)} \cdot \max_{\lambda \in \R^m} |\{x \in \{0,1\}^n: \sum_{i=1}^n x_i |(A^\T \lambda - c)_i| \leq \mathsf{IPGAP}(A,b,c)\}| + 1. \label{eq:bnb-size}
\end{equation}
\end{theorem}

In the second part of the framework, one leverages the randomness in the
coefficients of $A,c$ to upper bound the maximum size of any knapsack
in~\eqref{eq:bnb-size}. In~\cite{dey_branch-and-bound_2021}, they give such
an upper bound for the specific packing instances studied by Dyer and
Frieze~\cite{dyer_probabilistic_1989}. In the present work, we generalize
their probabilistic framework to random \emph{logconcave} IPs. We now state our
main meta-theorem, which we prove in~\Cref{sec:bnb-adapt}. 
 
 \begin{theorem}
 \label{thm:meta-logcon} 
 Let $n \geq 100(m+1)$, $b \in \R^m$, and $W := \begin{bmatrix} c^\T \\ A \end{bmatrix} \in
 \R^{n \times (m+1)}$ be a matrix whose columns are independent logconcave
 random vectors with identity covariance. Then, for $G
 \geq 0$, $\delta \in (0,1)$, with probability at least 
\[
1-\Pr_{A,c}[{\rm IPGAP}(A,b,c) \geq G]-\delta-e^{-n/5}, 
\]
the best bound first branch-and-bound algorithm applied to~\eqref{primal-ip}
produces a tree of size at most 
\begin{equation}
 n^{O(m)} e^{2\sqrt{2nG}}/\delta. \label{eq:final-bnb}
\end{equation}
\end{theorem}

The class of logconcave distributions is quite rich (see
subsection~\ref{sec:logcon} for a formal definition), e.g.~the uniform
distribution over any convex body as well as all of its marginals are
logconcave. We are therefore hopeful that interesting bounds on the size of
branch-and-bound trees can be obtained for a wide range of random logconcave
IPs, which by~\Cref{thm:meta-logcon} reduces to obtaining suitable bounds on
the integrality gap.

When $A,c$ have i.i.d.~uniform $[0,1]$ coefficients
and $b = \beta n$, $\beta \in
(0,1/2)^m$ and $\beta_{\min} := \min_{i \in [n]}
\beta_i$, Dyer and Frieze~\cite{dyer_probabilistic_1989} proved that for $n$ large enough 
\[
\Pr_{A,c}[{\rm IPGAP}(A,b,c) \geq \alpha a_1 \log^2 n/n] \leq 2^{-\alpha/a_2}+1/(2n), \forall \alpha \geq 1,
\]
where $a_1 = \Theta(1/\beta_{\min})^m$ and $a_2 = 2^{\Theta(m)}$.
In~\cite{dey_branch-and-bound_2021}, Dey, Dubey and Molinaro use this
integrality gap result combined with a probabilistic analysis of the bound
in~\Cref{thm:bnb} to show that the tree size is at most
\[
n^{O(ma_1\log a_1 + \alpha a_1 \log m)}
\]
with probability $1 - 2^{-\alpha/a_2} - 1/n$. A stronger bound can be obtained from \Cref{thm:meta-logcon}. 

We first observe that $2\sqrt{3} {\rm IPGAP}(A,b,c) = {\rm IPGAP}(2\sqrt 3 A, 2\sqrt 3 b, 2\sqrt 3 c)$,
noting that $W = 2\sqrt 3\begin{bmatrix} c^\T \\ A \end{bmatrix}$ has identity covariance.
Plugging $G = 2\sqrt 3 \alpha a_1 \log^2 n /n$ into \Cref{thm:meta-logcon} with $\delta = 1/(2n)-e^{-n/5}$, we get an improved tree-size bound of
\[
    n^{O(m)} e^{2\sqrt{4\sqrt 3 \alpha a_1 \log^2 n}} = n^{O(m) + 4\sqrt{\sqrt{3}\alpha a_1}}.
\]

Proceeding in a similar fashion, we can easily derive a tree-size bound for
Gaussian IPs by combining~\Cref{thm:ipgap_main_result}
and~\Cref{thm:meta-logcon}. 

\begin{corollary} \label{cor:gaussian-bnb}
For $C \geq 200$ as in~\Cref{thm:ipgap_main_result}, $m \in \N$, $n \geq C
m^{4.5}$, $A \in \R^{m \times n}, c \in \R^n$ with i.i.d. $\mathcal{N}(0,1)$
entries and $b \in \R^m$, $\|b^-\|_2 \leq n/10$. Then, for $1 \leq t \leq
\frac{n}{Cm^{2.5}(m + \log n)^2}$, with probability at least
$1-4(1-\frac{1}{25})^t-2/n^7$, the size of any best bound first
branch-and-bound tree for solving \eqref{primal-ip} is at most
$e^{O(\sqrt{t} m^{2.25})} n^{O(\sqrt{t} m^{1.25})}$.
\end{corollary}

\begin{proof}
Since $A,b,c$ satisfy the conditions of \Cref{thm:ipgap_main_result}, for
$G = 10^{15} \cdot \frac{m^{2.5}(m + \log n)^2}{n}$, we have that
$\mathrm{IPGAP}(A,b,c) \geq tG$ with probability at most
$4(1-\frac{1}{25})^t+1/n^7$. 
 
Applying \Cref{thm:meta-logcon} to $W$ with $\delta = 1/(2n^7)$, using the fact that $W \in \R^{m+1 \times n}$ has i.i.d.~$\mathcal{N}(0,1)$
entries, with probability at least 
\[
1-(4(1-\frac{1}{25})^t+1/n^7)-\delta-e^{-n/5} \geq 1-4(1-\frac{1}{25})^t-2/n^7, 
\]
we get that the size of the branch-and-bound tree is at most
\[
 n^{O(m)} e^{2\sqrt{2 t G n}}/\delta \leq n^{O(m)} e^{O(\sqrt{t} m^{1.25}(m+\log n))}(2n^7) =
 e^{O(\sqrt{t} m^{2.25})}  n^{O(\sqrt{t} m^{1.25})}. \quad \qed
\]
\end{proof}

\subsection{Proof Overview for~\Cref{thm:ipgap_main_result}}
\label{sec:ipgap-overview}

Our proof strategy follows along similar lines to that of Dyer and
Frieze~\cite{dyer_probabilistic_1989}, which we now describe. In their
strategy, one first solves an auxiliary LP $\max c^\T x, Ax \leq b - \epsilon
1_m$, for $\epsilon > 0$ small, to get its optimal solution $x^*$, which is
both feasible and nearly optimal for the starting LP (proved by a simple
scaling argument), together with its optimal dual solution $u^* \geq 0$ (see
subsection~\ref{sec:dual-progr} for the formulation of the dual). From here,
they round down the fractional components of $x^*$ to get a feasible IP
solution $x' := \lfloor x^* \rfloor$. We note that the feasibility of $x'$
depends crucially on the packing structure of the LPs they work with, i.e.,
that $A$ has non-negative entries (which does not hold in the Gaussian
setting). Lastly, they construct a nearly optimal integer solution $x''$, by
carefully choosing a subset of coordinates $T \subset \{i \in [n]: x'_i =
0\}$ of size $O(\operatorname{poly}(m) \log n)$, where they flip the
coordinates of $x'$ in $T$ from $0$ to $1$ to get $x''$. The coordinates of
$T$ are chosen accordingly the following criteria. Firstly, the coordinates
should be \emph{very cheap} to flip, which is measured by the absolute value
of their \emph{reduced costs}. Namely, they enforce that $\len{c_i -
A_{\cdot,i}^\T u^*} = O(\log n/n)$, $\forall i \in T$. Secondly, $T$ is
chosen to make the \emph{excess slack} $\length{A(x^*-x'')}_\infty \leq
1/\operatorname{poly}(n)$, i.e., negligible. We note that guaranteeing the existence
of $T$ is highly non-trivial. Crucial to the analysis is that after
conditioning on the exact value of $x^*$ and $u^*$, the columns of $W :=
\begin{bmatrix} c^\T \\ A \end{bmatrix} \in \R^{(m+1) \times n}$ (the objective
extended constraint matrix) that are indexed by $N_0 := \{i \in [n]: x^*_i = 0\}$
are independently distributed subject to having negative reduced cost, i.e.,
subject to $c_i - A_{\cdot,i}^\T u^* < 0$ for $i \in N_0$ (see
\Cref{lem:independence}). It is the large amount of left-over randomness in
these columns that allowed Dyer and Frieze to show the existence of the
subset $T$ via a discrepancy argument (more on this below). Finally, given a
suitable $T$, a simple sensitivity analysis is used
to show the bound on the gap between $c^\T x''$ and the~\eqref{primal-lp}
value. This analysis uses the basic formula for the optimality gap between
primal and dual solutions (see~\eqref{eq:gap-formula} in
subsection~\ref{sec:dual-progr}), and relies upon bounds on the size of the
reduced costs of the flipped variables, the total excess slack and the norm
of the dual optimal solution $u^*$.

\paragraph{\bf Adapting to the Gaussian setting.}
As a first difference with the above strategy, we are able to work directly
with the optimal solution $x^*$ of the original LP without having to replace
$b$ by $b' := b-\epsilon 1_m$. The necessity of working with this more
conservative feasible region in the packing setting of~\cite{dyer_probabilistic_1989} is that
flipping $0$ coordinates of $x'$ to $1$ can only \emph{decrease} $b-Ax'$. In
particular, if the coordinates of $b-Ax' \geq 0$ are too small, it becomes
difficult to find a set $T$ that doesn't force $x''$ to be infeasible. By
working with $b'$ instead of $b$, they can ensure that $b-Ax' \geq \epsilon
1_m$, which avoids this problem. In the Gaussian setting, it turns out that
we have equal power to both increase and decrease the slack of $b-Ax'$, due
to the fact that the Gaussian distribution is symmetric about $0$. We are in
fact able to simultaneously fix both the feasibility and optimality error of
$x'$, which gives us more flexibility. In particular, we will be able to use
randomized rounding when we move from $x^*$ to $x'$, which will allow us to
start with a smaller initial slack error than is achievable by simply rounding $x^*$ down.

\paragraph{\bf The Discrepancy Lemma.} 
Our main quantitative improvement -- the reduction from an exponential to a
polynomial dependence in $m$ -- arises from two main sources. The first
source of improvement is a substantially improved version of a discrepancy
lemma of Dyer and Frieze~\cite[Lemma 3.4]{dyer_probabilistic_1989}. This
lemma posits that for any large enough set of ``suitably random'' columns in
$\R^m$ and any not too big target vector $D \in \R^m$, then with
non-negligible probability there exists a set containing half the columns
whose sum is very close to $D$. This is the main lemma used to show the
existence of the subset $T$, chosen from a suitably filtered subset of the
columns of $A$ in $N_0$, used to reduce the excess slack. The non-negligible
probability in their lemma was of order $2^{-O(m)}$, which implied that one
had to try $2^{O(m)}$ disjoint subsets of the filtered columns before having
a constant probability of success of finding a suitable $T$. In our improved
variant of the discrepancy lemma, we show that by
sub-selecting a $1/(2\sqrt{m})$-fraction of the columns instead of
$1/2$-fraction, we can increase the success probability to constant, with the
caveat of requiring a slightly larger set of initial columns. The formal statement of our improved discrepancy lemma is given below.

\begin{restatable}{lemma}{disclemma}
	\label{lemma:disc}
	For $k,m \in \N$, let $a = \lceil 2\sqrt{m} \rceil$ and $\theta > 0$ satisfy
	$\left(\frac{2\theta}{\sqrt{2\pi k}}\right)^m {ak \choose k} = 1$. Let $Y_1,\dots,Y_{ak} \in \R^m$ be i.i.d. random vectors with independent coordinates. For $k_0 \in \N, \gamma \geq 0, M > 0$, assume that $\forall i \in [m]$, $Y_{1,i}$ is a $(\gamma,k_0)$-Gaussian convergent continuous random variable with maximum density at most $M$. Then, if
	\[
	k \geq \max \{ (4\sqrt{m} + 2)k_0, 144m^{\frac{3}{2}}(\log M + 3), 150\,000(\gamma + 1)m^{\frac{7}{4}} \},
	\]
	for any vector $D \in \R^m$ with $\length{D}_2 \leq \sqrt{k}$ the following holds:
	\begin{equation}
	\Prob \left[ \exists K \subset [ak] : |K| = k, \length{(\sum_{j\in K} Y_{j})-D}_\infty \leq \theta \right] \geq \frac{1}{25}. \label{eq:disc-prob}
	\end{equation}
\end{restatable}

The notion of Gaussian convergence used above (see~\Cref{sec:local-limit} for a formal definition), quantifies the speed at which the density of
normalized sums of i.i.d.~random variables converges to the standard Gaussian
density. This definition will in fact enforce that the entries of all the
vectors in~\Cref{lemma:disc} have mean $0$ and variance $1$. Apart from the
increased probability of success, we improve many other aspects
of~\cite[Lemma 3.4]{dyer_probabilistic_1989}. In particular, we remove the
restriction that the entries be bounded random variables, and we support
targets of norm exactly $\sqrt{k}$ instead of $k^{\alpha}$, for any $\alpha <
1/2$. Furthermore,~\cite[Lemma 3.4]{dyer_probabilistic_1989} is proved only
in the asymptotic regime where $k \rightarrow \infty$, whereas we give
explicit parameter dependencies, which are all polynomial in $m$. Taken
together, these improvements make the lemma easier to use and more flexible,
which should enable further applications. We refer the reader
to~\Cref{section:disc} for more details.

\paragraph{\bf Reduced cost filtering.} The second source of improvement is
the use of a much milder filtering step mentioned above. In both the uniform
and Gaussian case, the subset $T$ is chosen from a subset of $N_0$ associated
with columns of $A$ having reduced costs of absolute value at most some
parameter $\Delta > 0$. The probability of finding a suitable $T$ increases
as $\Delta$ grows larger, since we have more columns to choose from, and the
target integrality gap scales linearly with $\Delta$, as the columns we
choose from become more expensive as $\Delta$ grows. Depending on the
distribution of $c$ and $A$, the reduced cost filtering induces non-trivial
correlations between the entries of the corresponding columns of $A$, which
makes it difficult to use them within the context of the discrepancy lemma.
To deal with this problem in the uniform setting, Dyer and Frieze filtered
more aggressively, by additionally restricting to the columns of $A$ lying
in a sub-cube $[\alpha,\beta_{\min}]^m$, where $\alpha = \Omega(\log^3 n/n)$
and $\beta_{\min} := \min_{i \in [m]} \beta_i$ as above. This allowed them to ensure that the
distribution of the filtered columns in $A$ is uniform in
$[\alpha,\beta_{\min}]^m$, thereby removing the unwanted correlations. In the
packing setting, both aggressive and reduced cost filtering can have success
probability $\Theta(\beta_{\min})^m \Delta$, so aggressive filtering is not
much more expensive than reduced cost filtering. For an illustrative
calculation, if $u^* = 1_m/(m \beta_{\min})$ and $\Delta \leq 1$, then for $i \in N_0$, $A_i \in [0,1]^m, c_i \in [0,1]$ distributed uniformly, the reduced cost filtering probability is essentially equal to 
\begin{align*}
\Pr[A_i^\T u^*-c_i \in [0,\Delta]] &\leq \Delta \Pr[\sum_{j=1}^m A_{ij}/(m \beta_{\min}) \leq 2] \\
                                   &\leq \Delta {\rm vol}_m(\{a \geq 0: \sum_{j=1}^m a_j \leq 2m \beta_{\min}\}) \leq \Delta (2 e \beta_{\min})^m.
\end{align*}
When $A_i,c_i$ have $\mathcal{N}(0,1)$ entries however, we have $\Pr_{A_i,c_i}[A_i^\T u^*-c_i \in
[0,\Delta]] = \Theta(\Delta/\|(1,u^*)\|_2)$ for $\Delta \in [0,1]$. Given
this much larger success probability, we show how to work with only reduced
cost filtering in the Gaussian setting. While the entries of the filtered
columns of $A$ do indeed correlate, using the rotational symmetry of the
Gaussian distribution, we show that after applying a suitable rotation $R$,
the coordinates of the filtered columns of $RA$ are all independent (see \Cref{cor:rotate_coordinates}). This allows
us to apply the discrepancy lemma in a ``rotated space'', thereby completely
avoiding the correlation issues in the uniform setting.

\paragraph{\bf Sparsity of $x^*$ and boundedness of $u^*$.} As already mentioned, we are also able to substantially relax the rigid
requirements on the right hand side $b$ and to remove any stringent
``shape-dependence'' of the integrality gap on $b$. Specifically, the shape parameter $\beta_{\min}$ above is used to both lower bound $|N_0|$ by
roughly $\Omega((1-2\beta_{\min})n)$, the number of zeros in $x^*$, as well as upper
bound the $\ell_1$ norm of the optimal dual solution $u^*$ by $O(1/\beta_{\min})$ (this a main reason for the choice of the $[\alpha,\beta_{\min}]^m$ sub-cube above).
These bounds are both crucial for determining the existence of $T$. In the
Gaussian setting, we are able to establish $|N_0| = \Omega(n)$ and $\|u^*\|_2
= O(1)$, using only that $\|b^-\|_2 \leq n/10$. Due to the different nature
of the distributions we work with, our arguments to establish these bounds
are completely different from those used by Dyer and Frieze. Firstly, the
lower bound on $|N_0|$, which is strongly based on the packing structure of
the IP in~\cite{dyer_probabilistic_1989}, is replaced by a sub-optimality
argument. Namely, we show that the objective value of any LP basic solution
with too few zero coordinates must be sub-optimal, using the concentration
properties of the Gaussian distribution (see~\Cref{lem:ones}). The upper
bound on the $\ell_1$ norm of $u^*$ in~\cite{dyer_probabilistic_1989} is
deterministic and based on packing structure; namely, that the objective
value of a~\ref{primal-lp} of packing-type is at most $\sum_{i=1}^n c_i \leq
n$ (since $c_i \in [0,1], \forall i \in [m]$). In the Gaussian setting, we
prove our bound on the norm of $u^*$ by first establishing simple upper and
lower bounds on the dual objective function, which hold with overwhelming
probability, and optimizing over these simple approximations (see
\Cref{lem:primal-dual-props}).

\paragraph{\bf Future directions.} Given the above, a first question is
whether one can extend the integrality gap argument above to a larger class
of logconcave IPs. An important technical difficulty is to understand whether
\Cref{lemma:disc} can be generalized to handle random columns whose entries
are allowed to have non-trivial correlations and whose entries have non-zero
means. A second question is whether one can improve the current parameter
dependencies, both in terms of improving the integrality gap and relaxing the
restrictions on $b$. For this purpose, one may try to leverage flipping both
$0$s to $1$ and $1$s to $0$ in the rounding of $x'$ to $x''$.  The columns of
$W$ associated with the one coordinates of $x^*$ are no longer independent
however. A final open question is whether these techniques can be extended to
handle discrete distributions on $A$ and $c$.

\subsection{Related Work}
The worst-case complexity of solving $\max\{ c^\T x : Ax = b, x \geq 0, x \in
\Z^n\}$ scales as $n^{O(n)}$ times a polynomial factor in the bit complexity
of the problem. This is a classical result due to Lenstra \cite{lenstra1983}
and Kannan \cite{kannan1987} which is based on lattice basis reduction
techniques.

Beyond these worst-case bounds, the performance of basis reduction techniques
for determining the feasibility of random integer programs has been
analyzed. In this context, basis reduction is used to reformulate $Ax \leq b,
x \in \Z^n$ as $AUw \leq b, w \in \Z^n$ for some unimodular matrix $U \in
\Z^{n \times n}$, after which a simple variable branching scheme is applied
(i.e., branching on integer hyperplanes in the original space). Furst and
Kannan~\cite{FK89} showed that subset-sum instances of the form $\sum_{i=1}^n
x_i a_i = b$, $x \in \{0,1\}^n$, where each $a_i$, $i \in [n]$, is chosen
uniformly from $\{1,\dots,M\}$ and $b \in \Z_+$, can be solved in polynomial
time with high probability in this way if $M = 2^{\Omega(n^2)}$. Pataki,
Tural and Wong~\cite{pataki_basis_2010} proved generalizations of this result
for IPs of the form $f \leq Ax \leq g, l \leq x \leq u, x \in \Z^n$, where
the coefficients of $A$ are uniform in $\{1,\dots,M\}$ and $M$ is ``large''
compared to $\|(g-f,u-l)\|$. Apart from the different type of branching,
compared to the present work, we note that the IPs analyzed in these models
are either infeasible or have a unique feasible solution with high
probability.  

Another line of works has analyzed dynamic programming algorithm solving IPs with integer data~\cite{fritz2020,jansen2019,papadimitriou1981}. For $A \in \Z^{m \times n}, b \in \Z^m$,~\cite{jansen2019} proved that  
$ \max \{ c^\T x : Ax = b, x \geq 0, x \in \Z^n\}$
can be solved in time $O(\sqrt{m}\Delta)^{2m}\log(\|b\|_\infty)+O(nm)$,
where $\Delta$ is the largest absolute value of entries in the input matrix $A$ .
Integer programs of the form
$ \max \{ c^\T x : Ax = b, 0 \leq x \leq u, x \in \Z^n \}$
can similarly be solved in time
$$n \cdot O(m)^{(m+1)^2} \cdot O(\Delta)^{m\cdot(m+1)} \log^2(m\cdot \Delta),$$
which was proved in \cite{fritz2020}.
Note that integer programs of the form
$\max\{ c^\T x : Ax \leq b, x \in \{0,1\}^n\}$
can be rewritten in this latter form by adding $m$ slack variables.

The complexity of integer programming has also been studied from the perspective of \emph{smoothed analysis}. In this context, R\"oglin and V\"ocking~\cite{Rglin2007} proved that a class of IPs satisfying some minor conditions
has polynomial smoothed complexity
if and only if that class admits a pseudopolynomial time algorithm.
An algorithm has polynomial smoothed complexity if its running time is polynomial
with high probability when its input has been perturbed by adding random noise,
where the polynomial may depend on the inverse magnitude $\phi^{-1}$ of the noise
as well as the dimensions $n,m$ of the problem.
An algorithm runs in pseudopolynomial time if the running time is polynomial
when the numbers are written in unary, i.e., when the input data consists of
integers of absolute value at most $\Delta$ and the running time is bounded
by a polynomial $p(n, m, \Delta)$.
In particular, they prove that solving the randomly perturbed problem requires only polynomially
many calls to the pseudopolynomial time algorithm
with numbers of size $(nm\phi)^{O(1)}$
and considering only the first $O(\log(nm\phi))$ bits of each of the perturbed
entries.

One may in fact compare the complexity of dynamic programming and
branch-and-bound for Gaussian IPs using the result of~\cite{Rglin2007}. If we
choose $A \in \R^{m \times n},c \in \R^n$ as well as $b \in \R^m$ to have
i.i.d.~$N(0,1)$ entries, the result of \cite{Rglin2007} implies that with
high probability, to solve~\eqref{primal-ip} it is sufficient to solve
polynomially many problems with integer entries of size $n^{O(1)}$. Since
$\Delta = n^{O(m)}$ in this setting (by Hadamard's inequality), the result of
\cite{fritz2020} implies that~\eqref{primal-ip} can be solved in time
$n^{O(m^3)}$ with high probability. In comparison,
by~\Cref{cor:gaussian-bnb}, for any fixed $\epsilon \in (0,1)$,
branch-and-bound solves~\eqref{primal-ip} in time $n^{O(m^{1.25})}$, for $n
\geq 2^m$, with probability $1-\epsilon$. 

\subsection{Organization}
In \Cref{sec:preliminaries}, we give preliminaries on probability theory,
linear programming and integer rounding.
In \Cref{sec:properties_optimal_solutions}, we prove properties of the optimal primal and dual LP solutions $x^*$ and $u^*$,
and in \Cref{sec:independence_zeros}, we characterize the distribution of the columns of the objective extended constraint matrix 
corresponding to the zero entries of $x^*$.
In \Cref{sec:proof_main_result}, we prove \Cref{thm:ipgap_main_result},
using a discrepancy result that we prove in
\Cref{section:disc}. In \Cref{sec:bnb-adapt}, we prove \Cref{thm:meta-logcon}, our meta-theorem for random logconcave IPs.

\section{Preliminaries}
\label{sec:preliminaries}
\subsection{Basic Notation}

We denote the reals and non-negative reals by $\R,\R^+$ respectively, and the integers and positive integers by $\Z,\N$ respectively. For $k \geq 1$ an integer, we let $[k] := \{1,\dots,k\}$. For
$s \in \R$, we let $s^+ := \max\{s,0\}$ and $s^- := \min \{s,0\}$ denote the
positive and negative part of $s$. We extend this to a vector $x \in \R^n$ by
letting $x^{+(-)}$ correspond to applying the positive (negative) part
operator coordinate-wise. We let $\|x\|_2 = \sqrt{\sum_{i=1}^n x_i^2}$ and
$\|x\|_1 = \sum_{i=1}^n |x_i|$ denote the $\ell_2$ and $\ell_1$ norm
respectively. We use $\log x$ to denote the base $e$ natural logarithm. We use
$0_m,1_m \in \R^m$ to denote the all zeros and all ones vector respectively,
and $e_1,\dots,e_m \in \R^m$ denote the standard coordinate basis.
We write $\R^m_+ := [0,\infty)^m$.

For a random variable $X \in \R$, we let $\E[X]$ denote its expectation and
$\Var[X] := \E[X^2]-\E[X]^2$ denote its variance. For a random vector $X
\in \R^d$, we define its mean $\E[X] := (\E[X_1],\dots,\E[X_d])$ and its
covariance matrix 
\[
\Cov(X) := \E[XX^\T]-\E[X]\E[X]^\T = (\E[X_iX_j]-\E[X_i]\E[X_j])_{i,j \in [d]}.
\]
For any $u \in \R^d$, we note that $\Var[u^\T X] = \E[(u^\T X)^2]-\E[u^\T
X]^2 = u^\T \Cov(X) u$. We say that $X$ has identity covariance if $\Cov(X) =
I_d$, the $d \times d$ identity matrix. 

$X \in  \R^d$ is a continuous random vector if it admits a probability
density $f: \R^d \rightarrow \R_+$ satisfying $\Pr[X \in A] = \int_A f(x)
dx$, for all measurable $A \subseteq \R^d$. We will say that a continuous
random vector has maximum density at most $M > 0$ if its probability density
$f$ satisfies $\sup_{x \in \R^d} f(x) \leq M$. 

\subsection{The Dual Program, Gap Formula and the Optimal Solutions}
\label{sec:dual-progr}

A convenient formulation of the dual of \eqref{primal-lp} is given by
\begin{align}
    \min         & \ \text{val}^\star(u) :=b^\T u+\sum_{i=1}^n(c-A^\T u)_i^+ \tag{Dual LP} \label{dual-lp} \\
	\text{s.t. } & u\geq 0. \nonumber
\end{align}
To keep the notation concise, we will often use the identity $\|(c-A^tu)^+\|_1
= \sum_{i=1}^n(c-A^\T u)_i^+$.

For any primal solution $x$ and dual solution $u$ to the above pair of programs,
we have the following standard formula for the primal-dual gap:
\begin{align}
\text{val}^\star(u) - \text{val}(x)
   &:= b^\T u + \sum_{i=1}^n (c-A^\T u)_i^+ - c^\T  x \label{eq:gap-formula} \tag{Gap Formula} \\
   &= (b-Ax)^\T  u + \left(\sum_{i=1}^n x_i(A^\T u-c)_i^+ + (1-x_i)(c-A^\T u)_i^+ \right). \nonumber
\end{align}

Throughout the rest of the paper, we let $x^*$ and $u^*$ denote primal and
dual optimal basic feasible solutions for~\eqref{primal-lp} and~\eqref{dual-lp}
respectively, which we note are unique with probability $1$. We use the notation
\begin{equation}
\label{eq:obj-ext-mat}
W := \begin{bmatrix}
	c^\T \\
	A
    \end{bmatrix} \in \R^{(m+1) \times n},
\end{equation}
to denote the objective extended constraint matrix. We will frequently make
use of the sets $N_b := \{i \in [n]: x_i^* = b\}$, $b \in \{0,1\}$, the $0$
and $1$ coordinates of $x^*$, and $S := \{i \in [n]: x_i^* \in (0,1)\}$, the
fractional coordinates of $x^*$. We will also use the fact that $|S| \leq m$,
which follows since $x^*$ is a basic solution to~\eqref{primal-lp} and $A$
has $m$ rows.

\subsection{Chernoff Bounds and Binomial Sums}

Let $X_1,\dots,X_n$ independent $\{0,1\}$ random variables with $\mu = \E[\sum_{i=1}^n X_i]$. Then, the Chernoff bound gives \cite[Corollary 1.10]{doerr}
\begin{align}
\Prob[\sum_{i=1}^n X_i \leq \mu(1-\epsilon)] &\leq e^{-\frac{\epsilon^2 \mu}{2}}, \epsilon \in [0,1]. \label{eq:chernoff} \\
\Prob[\sum_{i=1}^n X_i \geq \mu(1+\epsilon)] &\leq e^{-\frac{\epsilon^2 \mu}{3}}, \epsilon \in [0,1]. \nonumber
\end{align}

The same concentration holds for the size of the intersection of two random sets.
\begin{lemma}
\label{lem:intersection}
Let $K, K'$ be two i.i.d. random subsets of $[ak]$, such that $|K| = |K'| = k$ and where $\Prob[i \in K] = \frac{1}{a}$ for every $1 \leq i \leq ak$. Then for every $\epsilon \in (0, 1)$ we have
$$\Prob \left[ \lenfit{K \cap K'} \geq \frac{(1+\epsilon)k}{a} \right] \leq 2\exp \left(-\frac{k\epsilon^2}{3a} \right),$$
$$\Prob \left[ \lenfit{K \cap K'} \leq \frac{(1-\epsilon)k}{a} \right] \leq 2\exp \left(-\frac{k\epsilon^2}{2a} \right).$$
\end{lemma}

\begin{proof}
The set size $|K \cap K'|$ follows a hypergeometric distribution.
To see this, we let $K \subset [ak]$ denote the set of successes,
and we sample $|K'|$ elements from $[ak]$ without replacement.
Then $|K \cap K'|$ counts the number of successes.
The bound now follows directly from \cite[Theorem 1.17]{doerr}.\qed
\end{proof}

We will need the following standard upper bound on binomial sums. For $n \geq 1$
and $n/2 \leq k \leq n$, we have that
\begin{equation}
\label{eq:bin-sum}
|\{S \subseteq [n]: |S| \geq k\}| = \sum_{i=k}^n \binom{n}{i} \leq e^{nH(k/n)},
\end{equation}
where $H(x) = -x\log(x) - (1-x)\log(1-x)$, $x \in [0,1]$, is the base $e$
entropy function \cite[Theorem 3.1]{galvin_three_2014}. We recall that $H(x)$ is concave in $x$ and $H(x) =
H(1-x)$, and hence is maximized at $H(1/2) = \log 2$.

\subsection{Bounds on the Moment Generating Function}

\begin{lemma}
\label{lem:mgf-cosh}
Let $Z \in \R$ be a random variable satisfying $\E[Z]=0$ and $\E[e^{|Z|}] <
\infty$. Then, $\E[e^{Z}] \leq \E[\cosh(\sqrt{3/2}Z)]$, where $\cosh(x) :=
\frac{1}{2}(e^x+e^{-x}) = \sum_{k=0}^\infty \frac{x^{2k}}{(2k)!}$.
\end{lemma}
\begin{proof}
\begin{align*}
\E[e^{Z}] &= \sum_{k=0}^\infty \frac{\E[Z^k]}{k!} \quad \left(\text{ by dominated convergence }\right)\\
          &= 1 + \sum_{k=1}^\infty \frac{\E[Z^{2k}]}{(2k)!} + \sum_{k=1}^\infty \frac{\E[Z^{2k+1}]}{(2k+1)!} \quad \left(~\E[Z]=0~\right) \\
          &\leq 1 + \sum_{k=1}^\infty \frac{\E[Z^{2k}]}{(2k)!} + \sum_{k=1}^\infty \frac{(\E[Z^{2k}] \E[Z^{2k+2}])^{1/2}}{(2k+1)!} \quad \left(\text{ by H{\"o}lder's inequality }\right) \\
          &\leq 1 + \sum_{k=1}^\infty \frac{\E[Z^{2k}]}{(2k)!} + \sum_{k=1}^\infty \frac{(2k+1)\E[Z^{2k}] + \E[Z^{2k+2}]/(2k+1)}{2(2k+1)!} \quad \left(\text{ by AM-GM }\right) \\
          &= 1 + \frac{3}{2} \frac{\E[Z^2]}{2} + \sum_{k=2}^\infty \frac{\E[Z^{2k}]}{(2k)!}(1 + 1/2 + \frac{2k}{2(2k-1)}) \\
          &\leq \sum_{k=0}^\infty \frac{\E[(\sqrt{3/2}Z)^{2k}]}{(2k)!} = \E[\cosh(\sqrt{3/2}Z)].
\end{align*}
The last inequality follows by the fact that $(3/2)^k \geq 3/2 + \frac{k}{2k-1}$ whenever $k \geq 2$. \qed
\end{proof}

\subsection{Gaussian and Sub-Gaussian Random Variables}
\label{sec:subg}

The standard, mean zero and variance $1$, Gaussian $\mathcal{N}(0,1)$ has density
function $\varphi(x) := \frac{1}{\sqrt{2\pi}}e^{-x^2/2}$. A standard Gaussian
vector in $\R^d$, denoted $\mathcal{N}(0,I_d)$, has probability density
$\prod_{i=1}^d \varphi(x_i) = \frac{1}{\sqrt{2\pi}^d} e^{-\|x\|^2/2}$
for $x \in \R^d$. A random variable $Y \in \R$ is $\sigma$-sub-Gaussian if for all $\lambda \in \R$, we have
\begin{equation}
\label{eq:gauss-mgf}
\E[e^{\lambda Y}] \leq e^{\sigma^2\lambda^2/2}.
\end{equation}
A standard normal random variable $X \sim \mathcal{N}(0,1)$ is $1$-sub-Gaussian. If variables
$Y_1,\dots,Y_k \in \R$ are independent and respectively
$\sigma_i$-sub-Gaussian, $i \in [k]$, then $\sum_{i=1}^k Y_i$ is
$\sqrt{\sum_{i=1}^k \sigma_i^2}$-sub-Gaussian.

\noindent For a $\sigma$-sub-Gaussian random variable $Y \in \R$ we have the
following standard tailbound:
\begin{equation}
\label{eq:gauss-tail}
\max \{\Prob[Y \leq -\sigma s], \Prob[Y \geq \sigma s]\} \leq e^{-\frac{s^2}{2}}, s \geq 0.
\end{equation}
For $X \sim \mathcal{N}(0,I_d)$, we will use the following higher dimensional
analogue:
\begin{equation}
\label{eq:gauss-norm}
\Prob[\|X\|_2 \geq s\sqrt{d}] \leq e^{-\frac{d}{2}(s^2- 2\log s-1)} \leq e^{-\frac{d}{2}(s-1)^2}, s \geq 1.
\end{equation}
We will use this bound to show that the columns of $A$ corresponding to the
fractional coordinates in the (almost surely unique) optimal solution $x^*$ are
bounded.
\begin{lemma}
	\label{lem:basis_bounded}
	Letting $S := \{i \in [n]: x_i^* \in (0,1)\}$, we have that
\[
\Prob[\exists i\in S: \|A_{\cdot, i}\|_2\geq (4\sqrt{\log(n)}+\sqrt{m})]\leq n^{-7}.\]
\end{lemma}

\begin{proof}
Using the union bound over all $n$ columns and \Cref{eq:gauss-norm} we get that:
\begin{align*}
    \Prob[\exists i\in S: \|A_{\cdot, i}\|_2 \geq (4\sqrt{\log(n)}&+\sqrt{m})] \\ &\leq \Prob[\exists i\in [n]: \|A_{\cdot, i}\|_2\geq (4\sqrt{\log(n)}+\sqrt{m})] \\
&\leq n\Prob_{X\sim \mathcal{N}(0,I_d)}[\|X\|_2\geq (4\sqrt{\log(n)}+\sqrt{m})] \\
&\leq n\exp\left(-\frac{m}{2}(4\sqrt{\log (n)/m})^2\right)=n^{-7}. \quad \qed
\end{align*}
\end{proof}

We will need the following analogue of the Chernoff bound for
truncated Gaussian sums.

\begin{lemma}
\label{lem:gauss-chernoff}
Let $X_1,\dots,X_n$ be i.i.d. $\mathcal{N}(0,1)$. Then
\[
\Prob \left[ \lenfit{ \sum_{i=1}^n X_i^+-\frac{n}{\sqrt{2\pi}} } \geq \sqrt{2n}s \right] \leq 2e^{-s^2/2}, s \geq 0.
\]
\end{lemma}
\begin{proof}
For $X \sim \mathcal{N}(0,1)$, a direct computation yields $\mu := \E[X^+] =
\frac{1}{\sqrt{2\pi}}$. For $\lambda \in \R$, we get that
\begin{align*}
\E[e^{\lambda(X^+-\mu)}]
&\leq \E[\cosh(\sqrt{3/2}\lambda(X^+-\mu))] \quad \left(\text{ by~\Cref{lem:mgf-cosh} }\right) \\
&\leq \E[\cosh(\sqrt{3/2}\lambda(X-\mu))] \quad \left(~|X^+-\mu| \leq |X-\mu|~\right) \\
&\leq e^{\frac{3}{2} \lambda^2/2} \cosh(\sqrt{\frac{3}{2}} \lambda \mu)  \quad \left(\text{ by~\eqref{eq:gauss-mgf} }\right) \\
&\leq e^{(1+\mu^2)\frac{3}{4}\lambda^2} \quad \left(~\cosh(x) \leq e^{x^2/2}~\right) \\
&\leq e^{\lambda^2} \quad \left(~\mu^2 = \frac{1}{2\pi} \leq \frac{1}{3}~\right).
\end{align*}
By the above, we have that $X^+-\frac{1}{2\pi}$ is $\sqrt{2}$-sub-Gaussian.
Therefore, if $X_1,\dots,X_n$ are i.i.d. $\mathcal{N}(0,1)$, the random
variable $Y = (\sum_{i=1}^n X_i^+)-\frac{n}{\sqrt{2\pi}} = \sum_{i=1}^n
(X_i^+-\frac{1}{\sqrt{2\pi}})$ is $\sqrt{2n}$-sub-Gaussian. The desired result
now follows directly from the sub-Gaussian tail bound~\eqref{eq:gauss-tail}
and the union bound. \qed
\end{proof}

\subsection{Logconcave Distributions}
\label{sec:logcon}

A probability measure $\mu$ on $\R^d$ is logconcave $\mu$ admits a
probability density $f: \R^d \rightarrow \R_+$ such that $\log f: \R^d
\rightarrow \R \cup \{-\infty\}$ is concave. We say that a random vector $X
\in \R^d$ is logconcave if its distributed according to a logconcave
probability measure. Important examples of logconcave probability
distributions are the Gaussian distribution and the uniform distribution on a
compact convex set. 

Logconcave distributions have many useful analytical properties. In
particular, the marginals of logconcave random vectors are also logconcave. 

\begin{theorem}[\cite{Prekopa71}]
\label{thm:log-marg}
Let $X \in \R^d$ be a logconcave random vector. Then, for any surjective
linear transformation $T: \R^d \rightarrow \R^k$, $TX$ is a logconcave random vector. 
\end{theorem}

The following theorem, which combines results from~\cite{LV07,Fradelizi99},
yields important properties of 1 dimensional logconcave distributions that we
will need. 

\begin{theorem}
\label{thm:log-main}
Let $\omega \in \R$ be a logconcave random variable with $\Var[\omega] = 1$.
\begin{itemize}
\item~\cite[Lemma 5.7]{LV07}: $\Pr[|\omega-\E[\omega]| \geq s] \leq
e^{-s+1}$, $\forall s \geq 0$.
\item~\cite[Theorem 4]{Fradelizi99}: $\omega$ has maximum density at most $1$.
\end{itemize}
\end{theorem}

\subsection{A Local Limit Theorem}
\label{sec:local-limit}

We now introduce the formalization of Gaussian convergence used
in~\Cref{lemma:disc}. 

\begin{definition}
Suppose $X_1, X_2, \ldots $ is a sequence of i.i.d copies of a random variable
$X$ with density $f$. For $k_0 \in \N$, $\gamma \geq 0$, we define $X$ to be $(\gamma, k_0)$-Gaussian convergent if
the density $f_n$ of $\sum\limits_{i=1}^n X_i / \sqrt{n}$ satisfies:
$$\len{f_n(x) -  \varphi(x)} \leq \frac{\gamma}{n} \quad \forall x \in \R, n
\geq k_0,$$
where $\varphi := \frac{1}{\sqrt{2\pi}}e^{-x^2/2}$ is the probability density function of the standard Gaussian.
\end{definition}

The above definition quantifies the speed of convergence in the context of
the central limit theorem. The rounding strategy used to obtain the main result utilizes random variables that are the weighted sum of a uniform and an independent normal variable. Crucially, the given convergence estimate will hold
for these random variables:

\begin{lemma}
\label{lem:unif-gaus-conv}
Let $U$ be uniform on $[-\sqrt{3},\sqrt{3}]$ and let $Z \sim \mathcal{N}(0,1)$. Then
there exists a universal constant $k_0 \geq 1$ such that $\forall
\epsilon \in [0,1]$, the random variable $\sqrt{\epsilon} U +
\sqrt{1-\epsilon} Z$ is $(1/10,k_0)$-Gaussian convergent and has maximum density at most $1$.
\end{lemma}

We note that the $O(1/n)$-convergence rate to Gaussian achieved above is due
to the first $3$ moments of $X := \sqrt{\epsilon} U + \sqrt{1-\epsilon} Z$
matching those of the standard Gaussian, namely $\E[X]=\E[X^3]=0$ and
$\E[X^2]=1$, and the fact that the characteristic function $c(s) :=
\E[e^{isX}]$ decays like $1/|s|$. More generally, if the first $l \geq
2$ moments match and the characteristic function decays quickly enough, the
convergence rate is $O(1/n^{(l-1)/2})$.
The proof follows from the following local limit theorem (a special case of~\cite[Theorem XVI.2.2]{Feller08}):

\begin{theorem}
\label{thm:local limit}
Let $X_1, X_2, \ldots$ be a sequence of i.i.d.~random variables with density
having moments $\E[X_1] = 0, \E[X_1^2] = 1,
\E[X_1^3] = 0$, $\E[X_1^4] = \mu_4 > 0$. Also, assume that we have $|\E[e^{isX}]| \leq
\beta/(|s|+1)^{\alpha}, \forall s \in \R$ for some $\beta, \alpha > 0$. Define $S_n =
\frac{1}{\sqrt{n}} \sum_{i=1}^n X_i$ and $\varphi(x) = \frac{1}{\sqrt{2\pi}}
\exp (-x^2/2).$ Then, for all $n \geq 1$, $S_n$ admits a density $f_n$ satisfying
$$f_n(x) = \varphi(x)\left[1 + \frac{\mu_4-3}{24n}(x^4-6x^2+3)\right] +  o\left(\frac{C_{\alpha,\beta}}{n}\right), \forall x \in \R,$$
where $C_{\alpha,\beta}$ depends only on $\alpha,\beta$.
\end{theorem}

\begin{proof}[\Cref{lem:unif-gaus-conv}]
Letting $X = \sqrt{\epsilon} U + \sqrt{1-\epsilon} Z$, a
straightforward calculation yields $\E[X]=\E[X^3]=0$, $\E[X^2]=1$ and
$\mu_4 := \E[X^4] = \epsilon^2 9/5 + 6(1-\epsilon)\epsilon + (1-\epsilon)^2
3$. Since convolution does not increase the maximum density, as $\max \{\sqrt{\epsilon},\sqrt{1-\epsilon} \} \geq 1/\sqrt{2}$, we see that the maximum density of $\sqrt{\epsilon} U + \sqrt{1-\epsilon}Z$ is upper bounded by $\min \{ \frac{1}{2\sqrt{3 \epsilon}}, \frac{1}{\sqrt{2\pi(1-\epsilon)}} \} \leq 1$ , where the terms correspond to the maximum density of $\sqrt{\epsilon} U$ and $\sqrt{1-\epsilon} Z$ respectively. Further,
\begin{align*}
|\E[e^{is(\sqrt{\epsilon} U + \sqrt{1-\epsilon} Z)}]| &=
|\frac{\sin(\sqrt{3\epsilon}s)}{\sqrt{3\epsilon}s}e^{-(1-\epsilon)s^2/2}| \leq \frac{2}{1+\sqrt{3\epsilon}|s|}\cdot\frac{1}{1+(1-\epsilon)s^2/2} \\ &\leq \frac{10}{1+|s|}.
\end{align*}
From here, $\max_{x \in \R} |\varphi(x)\frac{1}{24}(x^4-6x^2+3)|$ is
maximized at $x=0$ attaining value $\varphi(0)\frac{3}{24} =
\frac{1}{8\sqrt{2\pi}}$. Let $\gamma' := \frac{6}{40\sqrt{2\pi}} \geq
|\frac{(\mu_4-3)}{8\sqrt{2\pi}}|$. Applying~\Cref{thm:local limit}
with $\beta = 10,\alpha = 1$, one can choose $k_0 = O_{\alpha,\beta}(1)$
large enough so that the $o\left(C_{\alpha,\beta}/n\right)$ term is at most
$\gamma'/n$ for all $n \geq k_0$. The lemma now follows letting $\gamma := 1/10 \geq 2\gamma'$. \qed
\end{proof}

\subsection{Nets}

Let $\mathbb{S}^{d-1} = \{x \in \R^d: \|x\|_2=1\}$ denote the unit sphere in $\R^d$.
We say that $N \subseteq \mathbb{S}^{d-1}$ is an $\epsilon$-net if for every $x \in \mathbb{S}^{d-1}$
there exists $y \in N$ such that $\|x-y\| \leq \epsilon$. A classic
result we will need is that $S^{d-1}$ admits an $\epsilon$-net $N_\epsilon$ of
size $|N_\epsilon| \leq (1+\frac{2}{\epsilon})^d$. 
See, for example, \cite[Chapter 5]{vershynin}. We note that the same bound holds if we wish to construct a net of some subset $A \subseteq \mathbb{S}^{d-1}$ and we wish to have $N_\epsilon \subseteq A$.

\subsection{Rounding to Binary Solutions}
In the proof of \Cref{thm:ipgap_main_result}, we will take our optimal solution
$x^*$ and round it to an integer solution $x'$, by changing the fractional
coordinates. Note that as $x^*$ is a basic solution, it has at most $m$
fractional coordinates. One could round to a integral solution by setting all
of them to $0$, i.e., $x'=\lfloor x^*\rfloor$. If we assume that the Euclidean
norm of every column of $A$ is bounded by $C$, then we have
$\|A(x^*-x')\|_2 \leq mC$, since $x^*$ has at most $m$ fractional variables.
However, by using randomized rounding we can make this
bound smaller, as stated in the next lemma. We use this to obtain smaller
polynomial dependence in \Cref{thm:ipgap_main_result}.

\begin{lemma}
Consider an $m\times n$ matrix $A$ with $\|A_{\cdot, i}\|_2 \leq C$ for all $i\in [m]$ and
$y\in [0,1]^n$. Let $S = \{i \in [n]: y_i \in (0,1)\}$. There exists a vector $y'\in \{0,1\}^n$ with $\|A(y-y')\|_2 \leq C\sqrt{|S|}/2$ and $y'_i=y_i$ for all $i\notin S$.
\label{lem:round_integral_solution}
\end{lemma}

\begin{proof}
	Let $Y$ be the random variable in $\{0,1\}^n$ with independent components such that $\E(Y)=y$.
	Note that this implies that $\Var(Y_i)\leq 1/4$ for all $i$ and $\Var(Y_i)=0$ for $i\notin S$. Then:
	\begin{align*}
		\E(\|A(y-Y)\|_2)^2&\leq \E(\|A(y-Y)\|_2^2)= \sum_{i=1}^n\|A_{\cdot, i}\|_2^2\Var(Y_i)\\&\leq \sum_{i\in S}C^2\Var(Y_i) \leq \frac{C^2|S|}{4}
	\end{align*}
	So $\E(\|A(y-Y)\|_2)\leq C\sqrt{|S|}/2$, which directly implies the existence of a  value $y'\in \{0,1\}^n$ with $\|A(y-y')\|_2 \leq C\sqrt{|S|}/2$.\qed
\end{proof}

\section{Properties of the Optimal Solutions}
\label{sec:properties_optimal_solutions}
The following lemma is the main result of this section, which gives principal
properties we will need of the optimal primal and dual LP solutions. 
Namely, we prove an upper bound on the norm of the optimal dual solution
$u^*$ and a lower bound on the number of zero coordinates of the optimal
primal solution $x^*$.

\begin{lemma}
\label{lem:primal-dual-props}
Given $\delta :=
\frac{\sqrt{2\pi}}{n}\|b^-\|_2 \in [0,1/2)$, $\epsilon \in (0,1/5)$,
let $x^*, u^*$ denote the optimal primal and dual LP solutions, and let $\alpha :=
\frac{1}{\sqrt{2\pi}} \sqrt{\left( \frac{1-3\epsilon}{1-\epsilon} \right)^2-\delta^2}$ and choose $\beta \in [1/2,1]$ with $H(\beta) =
\frac{\alpha^2}{4}$. Then, with probability at least
$1-2\left(1+\frac{2}{\epsilon}\right)^{m+1} e^{-\frac{\epsilon^2
n}{8\pi}}-e^{-\frac{\alpha^2n}{4}}$, the following holds:
\begin{enumerate}
\item $c^\T x^* \geq \alpha n$.
\item $\|u^*\|_2 \leq \frac{1+\epsilon}{1-3\epsilon-(1-\epsilon)\delta}$.
\item $|\{i \in [n]: x^*_i = 0\}| \geq (1-\beta)n-m$.
\end{enumerate}
\end{lemma}

To prove~\Cref{lem:primal-dual-props}, we will require two technical
lemmas. The first, \Cref{lem:gauss-matrix}, shows that a random
Gaussian matrix forms a good embedding from $\ell_2$ into a ``truncated''
version of $\ell_1$ (i.e., $\ell_1$ restricted to non-negative coordinates). The
second, \Cref{lem:ones}, upper bounds the value of maximizing a Gaussian
objective over the hypercube restricted to vectors having a large number of
coordinates set to one.

\begin{lemma}
\label{lem:gauss-matrix}
Let $G \in \R^{n \times d}$ be a random matrix with independent $\mathcal{N}(0,1)$
entries. Then, for $\epsilon \in (0,1)$,
\begin{equation}
\Prob \left[\exists~v \in S^{d-1}, \|(Gv)^+\|_1 \notin \left[\frac{1-3\epsilon}{1-\epsilon},\frac{1+\epsilon}{1-\epsilon}\right] \cdot \frac{n}{\sqrt{2\pi}} \right] \leq 2\left(1+\frac{2}{\epsilon}\right)^d e^{-\frac{\epsilon^2 n}{8\pi}}.  \label{eq:l1-bad}
\end{equation}
\end{lemma}

\begin{proof}
Let $N_{\epsilon}$
denote an $\epsilon$-net of $S^{d-1}$. Let $B$ denote the event in
equation~\eqref{eq:l1-bad}. Let $E$ denote the event that there exists $v' \in
N_{\epsilon}$ such that $\|(Gv')^+\|_1 \not\in [1-\epsilon,1+\epsilon] \frac{n}{\sqrt{2\pi}}$.

We now show that $\Prob[B] \leq \Prob[E]$. For this purpose, it suffices to show that $\neg E \Rightarrow \neg B$. We thus condition $G$ on the
complement of $E$ and show that $B$ does not occur. For every $v \in S^{d-1}$, choose an $\tilde{v} \in
N_{\epsilon}$ satisfying $\|v-\tilde{v}\|_2 \leq \epsilon$. Then, we have
that
\begin{align}
\max_{v \in S^{d-1}} \|(Gv)^+\|_1 &\leq \max_{v \in S^{d-1}} \|(G\tilde{v})^+\|_1 + \|(G(v-\tilde{v}))^+\|_1 \nonumber \\
&\leq (1+\epsilon)\frac{n}{\sqrt{2\pi}} + \epsilon \max_{v \in S^{d-1}} \|(Gv)^+\|_1 \Rightarrow \nonumber \\
\max_{v \in S^{d-1}} \|(Gv)^+\|_1 &\leq \frac{1+\epsilon}{1-\epsilon} \frac{n}{\sqrt{2\pi}}. \label{eq:upper-plus}
\end{align}
To get lower bounds, we use that
\begin{align}
\min_{v \in S^{d-1}} \|(Gv)^+\|_1 &\geq \min_{v \in S^{d-1}} \|(G\tilde{v})^+\|_1 - \|(G(v-\tilde{v}))^-\|_1 \nonumber \\
&\geq (1-\epsilon)\frac{n}{\sqrt{2\pi}} - \epsilon \max_{v \in S^{d-1}} \|(Gv)^-\|_1 \nonumber \\
&= (1-\epsilon)\frac{n}{\sqrt{2\pi}} - \epsilon \max_{v \in S^{d-1}} \|(Gv)^+\|_1.
\intertext{From this inequality, we deduce that}
\min_{v \in S^{d-1}} \|(Gv)^+\|_1 &\geq \left( 1-\epsilon\left( 1 + \frac{1+\epsilon}{1-\epsilon} \right) \right) \frac{n}{\sqrt{2\pi}} = \frac{1-3\epsilon}{1-\epsilon} \frac{n}{\sqrt{2\pi}}. \label{eq:lower-plus}
\end{align}

Thus $\neg E \Rightarrow \neg B$, as needed. Using that $Gv \sim
\mathcal{N}(0,I_n)$ for $v \in S^{d-1}$, we have that
\begin{align*}
\Prob[E] &\leq \sum_{\tilde{v} \in N_{\epsilon}} \Prob \left[ \|(G\tilde{v})^+\|_1 \not\in [1-\epsilon,1+\epsilon] \cdot \frac{n}{\sqrt{2\pi}} \right] \\
         &= |N_{\epsilon}| \Prob_{X \sim \mathcal{N}(0,I_n)} \left[ \|X^+\|_1 \not\in [1-\epsilon,1+\epsilon] \frac{n}{\sqrt{2\pi}} \right] \\
         &\leq 2 \left( 1+\frac{2}{\epsilon} \right)^d e^{-\frac{\epsilon^2 n}{8\pi}} \quad \left(\text{ by~\Cref{lem:gauss-chernoff} }\right).
\end{align*}
The lemma thus follows. \qed
\end{proof}

\begin{lemma}
\label{lem:ones}
Let $c \sim \mathcal{N}(0,I_n)$ and $\alpha \in [0,2\sqrt{\log 2}]$. Choose $\beta
\in [1/2,1]$ satisfying $H(\beta) = \frac{\alpha^2}{4}$. Then
\begin{equation}
\Prob \left[\max_{x \in \{0,1\}^n, \|x\|_1 \geq \beta n} c^\T x \geq \alpha n \right] \leq e^{-\frac{\alpha^2 n}{4}}. \label{eq:ones-event}
\end{equation}
\end{lemma}

\begin{proof}
For $x \in \{0,1\}^n$, we have that $c^\T x \sim \mathcal{N}(0,\|x\|_1)$.  By the
sub-Gaussian tail bound~\eqref{eq:gauss-tail}, we see that
\[
\Prob[c^\T x \geq \alpha n] = \Prob_{z \in \mathcal{N}(0,1)}[\|x\|_2z \geq \alpha
n] \leq \Prob_{z \in \mathcal{N}(0,1)}[z \geq \alpha
\sqrt{n}] \leq  e^{-\frac{\alpha^2n}{2}}.
\]
By union bound and~\eqref{eq:bin-sum}, we conclude that
\begin{align*}
\Prob \left[\max_{x \in \{0,1\}^n, \|x\|_1 \geq \beta n} c^\T x \geq \alpha n \right]
&\leq |\{x \in \{0,1\}^n: \|x\|_1 \geq \beta n\}|\Prob_{z \in \mathcal{N}(0,1)}[z \geq \alpha \sqrt{n}] \\ &\leq e^{H(\beta)n}e^{-\frac{\alpha^2}{2}n} = e^{-\frac{\alpha^2n}{4}}, \text{ as needed.} \quad \qed
\end{align*}
\end{proof}

We are now ready to prove~\Cref{lem:primal-dual-props}.
\begin{proof}[\Cref{lem:primal-dual-props}]

Let $W \in \R^{(m+1) \times n}$ be as in~\eqref{eq:obj-ext-mat}. Applying~\Cref{lem:gauss-matrix} to
$(W^\T,\epsilon)$ and~\Cref{lem:ones} to $(c,\alpha,\beta)$, the events
($E_1$) $\|(W^\T u)^+\|_1 \in
\left[ \frac{1-3\epsilon}{1-\epsilon},\frac{1+\epsilon}{1-\epsilon} \right] \cdot
\frac{n}{\sqrt{2\pi}}$, $\forall \|u\|_2=1$, and ($E_2)$ $\max_{x \in
\{0,1\}^n, \|x\|_1 \geq \beta n} c^\T x < \alpha n$ hold with probability
at least $1-2\left(1+\frac{2}{\epsilon}\right)^{m+1} e^{-\frac{\epsilon^2
n}{8\pi}}-e^{-\frac{\alpha^2n}{4}}$. It thus suffices to prove that
properties 1-3 hold under $E_1$, $E_2$. We prove each in turn below.

\paragraph{\bf Proof of 1.} To lower bound $c^\T x^*$, i.e., the optimal LP value, it suffices to show that every dual solution has value at least $\alpha n$. Take $u \in \R^m_+$, then we can bound the dual value ${\rm val}^*(u^*)$ using the inequalities
\begin{align}
b^\T u + \|(c-A^\T u)^+\|_1 &\geq (b^{-})^\T u + \|(W^\T(1,-u)^\T)^+\|_1 \nonumber \\
                         &\geq -\|b^-\|_2 \|u\|_2 + \frac{1-3\epsilon}{1-\epsilon} \frac{n}{\sqrt{2\pi}} \sqrt{1+\|u\|_2^2} \quad \left(\text{ by $E_1$ }\right) \nonumber \\
                         &= \frac{n}{\sqrt{2\pi}}\cdot \left(-\delta\|u\|_2 + \frac{1-3\epsilon}{1-\epsilon}\sqrt{1+\|u\|_2^2}\right) \label{eq:dual-lb} \\
                         &\geq \frac{n}{\sqrt{2\pi}}\sqrt{(\frac{1-3\epsilon}{1-\epsilon})^2-\delta^2} = \alpha n \nonumber,
\end{align}
where the last inequality follows by minimizing the expression over
$\|u\|$, which occurs at $\|u\| =
\frac{\delta}{\sqrt{(\frac{1-3\epsilon}{1-\epsilon})^2-\delta^2}}$.

\paragraph{\bf Proof of 2.} To bound $\|u^*\|_2$, we compare bounds for ${\rm val}^*(u^*)$:
\begin{align*}
\frac{1+\epsilon}{1-\epsilon} \frac{n}{\sqrt{2\pi}} &\geq \|(W^\T(1,0_m)^\T)^+\|_1 = \|c^+\|_1 \quad \left(\text{by $E_1$}\right)\\
&=  {\rm val}^*(0_m) \geq {\rm val}^*(u^*) \\
&\geq \frac{n}{\sqrt{2\pi}}\cdot \left(-\delta\|u^*\|_2 + \frac{1-3\epsilon}{1-\epsilon}\sqrt{1+\|u^*\|_2^2}\right) \quad \left(\text{ by~\eqref{eq:dual-lb} }\right) \\
&\geq \frac{n}{\sqrt{2\pi}}\cdot \left(\frac{1-3\epsilon}{1-\epsilon}-\delta\right)\|u^*\|_2 \\
&\Rightarrow \|u^*\|_2 \leq \frac{1+\epsilon}{1-3\epsilon-(1-\epsilon)\delta}.
\end{align*}

\paragraph{\bf Proof of 3.} The optimal feasible solution $x^*$ is unique
almost surely, and as such it is a basic feasible solution.
As such, we know that $x^* \in [0,1]^n$ has at most $m$
fractional components. In particular, $|\{i \in [n]: x^*_i=0\}| \geq n-m-|\{i
\in [n]: x^*_i=1\}|$. Thus, it suffices to show that $|\{i \in [n]:
x^*_i=1\}| \leq \beta n$. Define $\bar{x} \in \{0,1\}^n$ satisfying
\[
\bar{x}_i = \begin{cases}
x^*_i&:~~ x^*_i \in \{0,1\} \\
1&:~~ x^*_i \in (0,1), c_i \geq 0 \\
0&:~~ x^*_i \in (0,1), c_i < 0
\end{cases}, \quad \quad \forall i \in [n].
\]
Clearly, $c^\T \bar{x} \geq c^\T x^* \geq \alpha n$. Thus, by $E_2$ we have that $\beta n > |\{i \in [n]: \bar{x}_i = 1\}| \geq |\{i \in [n]: x_i^*=1\}|$, as needed. \qed
\end{proof}

\section{Properties of the $0$ Columns}
\label{sec:independence_zeros}
For $Y := (c,a_1,\dots,a_m) \sim \mathcal{N}(0,I_{m+1})$ and $u \in \R^m_+$, let $Y^u$
denote the random variable $Y$ conditioned the event $c - \sum_{i=1}^m u_i
a_i \leq 0$. We will crucially use the following lemma directly adapted from Dyer and
Frieze~\cite[Lemma 2.1]{dyer_probabilistic_1989}, which shows that the columns
of $W$ associated with the $0$ coordinates of $x^*$ are independent subject to
having negative reduced cost.

Recall that $N_b = \{i \in [n] : x^*_i = b\}$,
that $S = \{i \in [n] : x^*_i \in (0,1)\}$ and that
$W := \begin{bmatrix}
	c^\T \\
	A
    \end{bmatrix} \in \R^{(m+1) \times n}$
is the objective extended constraint matrix.

\begin{lemma}
\label{lem:independence}
Let $N_0' \subseteq [n]$.
Conditioning on $N_0 = N_0'$, the submatrix $W_{\cdot, [n]\setminus N_0'}$
uniquely determines $x^*$ and $u^*$ almost surely.
If we further condition on the exact value of $W_{\cdot, [n]\setminus N_0'}$,
assuming $x^*$ and $u^*$ are uniquely defined,
then any column $W_{\cdot, i}$ with $i \in N_0'$ is distributed
according to $Y^{u^*}$ and independent of $W_{\cdot,[n]\setminus\{i\}}$.
\end{lemma}

\begin{proof}
    Knowing $N_0'$, we solve the following program to obtain
    its primal and dual optimal feasible solutions $\bar x$ and $\bar u$.
    \begin{align*}
        \max~&c^\T x \\
        \text{s.t.}~&\sum_{i \in [n]\setminus N_0} x_i a_{ji} \leq b_j && \forall j \in [m]\\
                    & x_i = 0 && \forall i \in N_0' \\
                    & x \in [0,1]^n.
    \end{align*}
    This does not require knowledge of $W_{\cdot,N_0'}$,
    and the optimal feasible primal and dual solutions are unique almost surely.

    If $N_0 = N_0'$, then $\bar x = x^*$ and $\bar u = u^*$.
    Since these solutions satisfy complementary slackness,
    this is equivalent to the following system of equations.
\begin{align}
(1-\bar x_i)(c_i-\sum_{j=1}^m \bar u_ia_{ji})^+ = 0, \label{eq:lem9:cond1}&&\forall i\in [n]. \\
\bar x_i(\sum_{j=1}^m \bar u_ia_{ji}-c_i)^+ = 0, \label{eq:lem9:cond2} &&\forall i\in [n] .\\
\bar u_j (b-A\bar x)_j = 0, &&\forall j \in [m]\label{eq:lem9:cond3}.
\end{align}

Note that for $i \in N_0$, \cref{eq:lem9:cond2,eq:lem9:cond3} are trivially satisfied.
By definition, the distribution
of $W_{\cdot,i} := (c_i,a_{1i},\dots,a_{mi})$ conditioned on
\cref{eq:lem9:cond1} for $\bar x_i=0$ has the same law as $Y^{\bar u}$. Note that each
of these conditions depends on only one $i\in N_0'$, so all columns of
$W_{\cdot, N_0'}$ are independent.

We conditioned only on $N_0 = N_0'$, which has non-zero probability, and we
have shown that for every possible realization of $W_{\cdot,[n]\setminus N_0'}$,
the columns of $W_{\cdot,N_0'}$ are independently distributed as
$Y^{u^*}$, which proves the lemma. \qed
\end{proof}

To make the distribution of the columns $A_{\cdot, i}$ easier to analyze we rotate them.
\begin{lemma}
\label{cor:rotate_coordinates}
 Let $R$ be a rotation that sends the vector $u$ to the vector $\|u\|_2e_m$.
 Suppose $(c, a) \sim Y^{u}$. Define $a':=Ra$. Then $(c,Ra)\sim (c',a')$, where $(c',a')$ is the value of $(\bar{c}',\bar{a}')\sim \mathcal{N}(0,I_{m+1})$ conditioned on $\|u\|_2\bar{a}'_m-\bar{c}'\geq 0$.
\end{lemma}

\begin{proof}
Recall that $(c,a)\sim Y^{u'}$ is generated by conditioning $(\bar{c},\bar{a})\sim\mathcal{N}(0,I_{m+1})$ on $\bar{c}-u^\T \bar{a}\leq 0$. The latter is equivalent to $\bar{c}-\|u\|_2(R \bar{a})_m=\bar{c}-u^\T R^{-1}(R \bar{a}) \leq 0$, i.e. to $\|u\|_2(R \bar{a})_m-\bar{c}\geq 0$. Setting $(\bar{c},'\bar{a}')=(\bar{c},R\bar{a})$, we see that $(c,Ra)\sim (c',a')$, where $(c',a')$ is the value of $(\bar{c}',\bar{a}')\sim \mathcal{N}(0,I_{m+1})$ conditioned on $\|u\|_2\bar{a}'_m-\bar{c}'\geq 0$. \qed
\end{proof}

We will slightly change the distribution of the $(c',a'_m)$ above using
rejection sampling, as stated in the next lemma. This will make it easier to
apply the discrepancy result of \Cref{lemma:disc}, which is used to round
$x^*$ to an integer solution of nearby value. In what follows, we denote the
probability density function of a random variable $X$ by $f_X$. In the
following lemma, we use ${\rm unif}(0,\nu)$ to denote the uniform
distribution on the interval $[0,\nu]$, for $\nu \geq 0$.

\begin{lemma}
	\label{lem:rejection_sampling}
        For any $\omega \geq 0$, $\nu > 0$,
        let $X,Y\sim \mathcal{N}(0,1)$ be independent random variables and let $Z=\omega Y-X$. Let $X',Y',Z'$ be these variables conditioned on $Z\geq 0$. We apply rejection sampling on $(X', Y', Z')$ with acceptance probability
\begin{align*}
	\Prob[\text{accept} | Z'=z]=\frac{2\varphi(\nu/\sqrt{1+\omega^2})\mathbf{1}_{z \in [0,\nu]}}{2\varphi(z/\sqrt{1+\omega^2})}.
\end{align*}
Let $X'', Y'', Z''$ be the variables $X',Y',Z'$ conditioned on acceptance. Then:
\begin{enumerate}
	\item $\Prob[\text{accept}]= 2\nu\varphi(\nu/\sqrt{1+\omega^2})/\sqrt{1+\omega^2}$.
	\item $Y''\sim W+V$ where $W\sim \mathcal{N}(0,\frac{1}{1+\omega^2})$, $V\sim \operatorname{unif}(0,\frac{\nu\omega}{1+\omega^2})$ and $W,V$ are independent.
\end{enumerate}
\end{lemma}
\begin{proof}
Because $Z\sim \mathcal{N}(0,1+\omega^2)$ and  $Z'=Z\mid Z\geq 0$, for the density function $f_{Z'}$ we have $f_{Z'}(z)=2\cdot \mathbf{1}_{z\geq 0}\varphi(z/\sqrt{1+\omega^2})/\sqrt{1+\omega^2}$.
\begin{align*}
\Prob[\text{accept}]= \int_0^\nu\frac{2\varphi(z/\sqrt{1+\omega^2})}{\sqrt{1+\omega^2}}\frac{2\varphi(\nu/\sqrt{1+\omega^2})}{2\varphi(z/\sqrt{1+\omega^2})}dz=\frac{2\nu\varphi(\nu/\sqrt{1+\omega^2})}{\sqrt{1+\omega^2}}.
\end{align*}
Now the probability density function of $Z''$, the variable $Z'$ conditioned on acceptance, is:
\begin{align*}
f_{Z''}(z)=\frac{f_{Z'}(z)\cdot \Prob[\text{accept}|Z'=z]}{\Prob[\text{accept}]}=\mathbf{1}_{z\in [0,\nu]}/\nu.
\end{align*}
So $Z''$ is uniformly distributed in $[0,\nu]$.

Let $W = Y-\frac{\omega}{1+\omega^2} Z$. By a direct calculation, one can check that
$\E[W]=\E[Z]=\E[WZ]=0$, $\E[Z^2] = 1+\omega^2$, $\E[W^2]=\frac{1}{1+\omega^2}$. Since
$W,Z$ are jointly Gaussian, this covariance structure implies that $W \sim
\mathcal{N}(0,\frac{1}{1+\omega^2})$, $Z \sim \mathcal{N}(0,1+\omega^2)$ and that $W,Z$ are independent.
Noting that $Y = W + \frac{\omega}{1+\omega^2}Z$, we
see that the law of $Y''$ is the same as that of
$W+\frac{\omega}{1+\omega^2}Z''$, where $W,Z''$ are independent. Finally observe that $\frac{\omega}{1+\omega^2}Z''\sim \text{unif}(0,\frac{\nu\omega}{1+\omega^2})$. \qed
\end{proof}

\section{Proof of~\Cref{thm:ipgap_main_result}}
\label{sec:proof_main_result}
Recall that $S = \{i \in [n]: x_i^* \in (0,1)\}$ and $N_0 = \{i \in [n]: x^*_i = 0\}$.
To prove \Cref{thm:ipgap_main_result}, we will need to condition on the following event, which we denote by $E$:
\begin{enumerate}
	\item $\|A_{\cdot, i}\|_ 2\leq 4\sqrt{\log (n)}+\sqrt{m}$, $\forall i\in S$.
	\item $\|u^*\|\leq 3$.
	\item $|N_0|\geq n/500$.
\end{enumerate}
Using \Cref{lem:primal-dual-props,lem:basis_bounded} we can show that
$E$ hold with probability $1-n^{-\Omega(1)}$.
Now we take our optimal basic solution $x^*$ and round it to an integral
vector $x'$ using \Cref{lem:round_integral_solution}.  Then we can generate a
new solution $x''$ from $x'$ by flipping the values at indices $T\subseteq N_0$
to one.
In \Cref{lem:high_prob_set_T} we show that with high probability there is such
a set $T$, such that $x''$ is a feasible solution to our primal problem and
that $\val(x^*)-\val(x'')$ is small.

We do this by looking at $t$ disjoint subsets of $N_0$ with small reduced costs. Then we show for each of these sets that with
constant probability it contains a subset $T$ such that for $x''$ obtained from
$T$, $x''$ is feasible and all constraints that are tight for $x^*$ are close
to being tight for $x''$. This argument relies on the~\Cref{lemma:disc} from the introduction, which we prove in \Cref{section:disc}.

If a suitable $T$ exists, then using the gap formula we show that $\val(x^*)-\val(x'')$ is small. Because the $t$ sets
independent the probability of failure decreases exponentially with
$t$. Hence, we can make the probability of failure arbitrarily small by
increasing $t$. We know $\val(x^*)=\val_\mathsf{LP}(A,b,c)$ and because $x''\in
\{0,1\}^n$ we have $\val_\mathsf{IP}(A,b,c) \geq \val(x'')$, so
$\mathsf{IPGAP}(A,b,c) =\val_\mathsf{LP}(A,b,c) -\val_\mathsf{IP}(A,b,c) \leq \val(x^*)-\val(x'')$, which
is small with high probability.
\begin{lemma}
	\label{lem:high_prob_set_T}
	For $n \geq \exp(k_0)$, with $k_0$ as in \Cref{lem:unif-gaus-conv}, we have that 
\begin{equation}
\Prob\left[\mathsf{IPGAP}(A,b,c) > \finalHpUpperBoundConstant  t\cdot \frac{m^{2.5}(\log n + m)^2}{n} \mid E \right] \leq 2\cdot \left(1-\frac{1}{25}\right)^{t} \label{eq:event_prob_bound}
\end{equation}
for $1 \leq t \leq \frac{ n}{20\,000\sqrt{m}k^2}$, where $k := \lceil 165\,000m(\log(n) +m)\rceil$.
\end{lemma}

\begin{proof}
To prove the lemma, we show that the desired probability bound holds when we
condition on the exact values of $N_0 \subseteq [n]$ and
$W_{\cdot,[n]\setminus N_0}$ subject to 1-3 defining $E$. Since $N_0$ and
$W_{\cdot,[n]\setminus N_0}$ determine $E$, this is clearly sufficient. By
\Cref{lem:independence}, we may further assume that this conditioning
uniquely determines $x^*$ and $u^*$. 

Now let $R$ be a rotation that sends the vector $u^*$ to the vector $\|u^*\|_2e_m$.
	Define:
	\begin{align*}
	\Delta&:= 10\,000\sqrt{m}k / n = \Theta(m^{1.5}(m + \log n)/n),  \\
	B_i&:=RA_{\cdot,i}, &\text{for } i\in N_0,\\
	Z_t&:=\{i\in N_0:\|u^*\|_2(B_i)_m-c_i \in [0,t\Delta ]\}, & \text{for } 1 \leq t\leq \frac{1}{2\Delta k}.
	\end{align*}
	 We consider a (possibly infeasible) integral solution $x'$ to the LP, generated by rounding the fractional coordinates from $x^*$. By \Cref{lem:round_integral_solution} we can find such a solution with $\|A(x^*-x')\|_2 \leq (4\sqrt{\log n} + \sqrt{m})\sqrt{|S|}/2\leq (4\sqrt{\log n} + \sqrt{m})\sqrt{m}/2$.
        We will select a subset $T \subseteq Z_t$ of size $|T|=k$ of coordinates
        to flip from $0$ to $1$ to obtain $x'' \in \{0,1\}^n$ from $x'$, so $x'' := x'+\sum_{i \in T} e_i$.
        By complementary slackness,
        we know for $i \in [n]$ that
        $x^*_i(A^\T u^* - c)^+_i = (1 - x^*_i)(c - A^\T u^*)^+ = 0$
        and that
        $x^*_i \notin \{0,1\}$ implies $(c - A^\T u^*)_i = 0$, and
        for $j \in [m]$ that $u^*_j > 0$ implies
        $b_j = (Ax^*)_j$.
        This observation allows us to prove the following key bound
        for the integrality gap of \eqref{primal-lp}
        \begin{align*}
	&\val(x^*)-\val(x'')=\text{val}^\star(u^*) - \text{val}(x'')
	\\&= (b-Ax'')^\T  u^* \\&+ \left(\sum_{i=1}^n x''_i(A^\T u^*-c)_i^+ + (1-x''_i)(c-A^\T u^*)_i^+ \right) & \text{(by~\eqref{eq:gap-formula})}
        \\&= (x^* - x'')^\T A^\T u^* + \sum_{i\in T} (A^\T u^*-c)_i & \hspace{-2cm}\text{(by complementary slackness)}\\
	&\leq \sqrt{m}\|u^*\|_2\left\|A(x''-x^*)\right\|_\infty  + t\Delta k&\text{(since $T\subseteq Z_t$)}.
	\end{align*}
        Condition 2 tells us that $\|u^*\|_2 \leq 3$, and by definition we have
        $$t\Delta k \leq 27226\cdot 10^{10} t \cdot \frac{m^{2.5}(\log(n) + m)^2}{n},$$
        so the rest of
        this proof is dedicated to showing the existence of a set $T \subseteq Z_t$
        such that $\|A(x''-x^*)\|_\infty \leq O(1/n)$ and $Ax'' \leq b$.

        By applying \Cref{lem:independence}, we see that $\{(c_i,
        A_{\cdot,i})\}_{i\in N_0}$ are independent vectors, distributed as
        $\mathcal{N}(0,I_{m+1})$ conditioned on $c_i-A_{\cdot,i}^\T u^*\leq 0$.
        This implies that the vectors $\{(c_i,B_i)\}_{i \in N_0}$, are also
        independent. By \Cref{cor:rotate_coordinates}, it follows that $(c_i,B_i)\sim
        \mathcal{N}(0,I_{m+1}) \mid \|u^*\|_2(B_i)_m -c_i \geq 0$. Note that
        the coordinates of $B_i$ are therefore independent and $(B_i)_j \sim
        \mathcal{N}(0,1)$ for $j\in [m-1]$.

        To simplify the upcoming calculations, we apply rejection sampling
        as specified in \Cref{lem:rejection_sampling} with $\nu=\Delta t$ on
        $(c_i,(B_i)_m)$, for each $i \in N_0$. Let $Z_t' \subseteq N_0$ denote
        the indices which are accepted by the rejection sampling procedure. By
        the guarantees of~\Cref{lem:rejection_sampling}, we have that $Z_t'
        \subseteq Z_t$ and \begin{align*}
\Prob[i\in Z_t'\mid i\in N_0]= \frac{2 \Delta t\varphi(\Delta t/\sqrt{1+\|u^*\|_2^2})}{\sqrt{1+\|u^*\|_2^2}} \geq \frac{2 \Delta t\varphi(1/2)}{\sqrt{10}}\geq  \Delta t /5.
	\end{align*}
Furthermore, for $i \in Z_t'$ we know that $(B_i)_m$ is distributed as a sum of
independent $N(0,\frac{1}{1+\|u^*\|_2^2})$ and ${\rm unif}(0,t\Delta)$
random variables, and thus $(B_i)_m$ has mean and variance
\begin{align*}
        \mu_t&:=\E[(B_i)_m|i\in Z_t'] = \Delta t / 2, \\
        \sigma^2_t&:=\Var[(B_i)_m|i\in Z_t']= \frac{1}{1+\|u^*\|_2^2}+\frac{1}{12}\left(\frac{\|u^*\|_2\Delta t}{1+\|u^*\|_2^2}\right)^2 \in [1/10, 2].
    \end{align*}

		Now define $\Sigma^{(t)}$ to be the diagonal matrix with $\Sigma^{(t)}_{j,j} = 1$, $j \in [m-1]$, and $\Sigma^{(t)}_{m,m} = \sigma_t$.
        Conditional on $i \in Z_t'$, define $B^{(t)}_i$ as the random variable
        \begin{align*}
            B^{(t)}_i &:= (\Sigma^{(t)})^{-1}(B_i-\mu_te_m) \mid i \in Z_t'.
    \end{align*}
    This ensures that all coordinates of $B^{(t)}$ are independent, mean zero and have variance one.

	We have assumed that $|N_0|\geq n/500$ and we know ${\Prob[i\in Z_t'|i\in N_0]\geq \Delta t /5}$.
	Now, using the Chernoff bound
        \eqref{eq:chernoff}
        we find that:
	\begin{align}
	\Prob[|Z_t'|< 2t\sqrt{m}k]&\leq  \Prob\left[|Z_t'|< \frac{1}{5}t\Delta |N_0|/2\right]\nonumber\\
        &\leq \exp\left(-\frac{1}{8} \cdot\frac{1}{5}t\Delta |N_0|\right)\nonumber\\
	&\leq \left(1-\frac{1}{25}\right)^{t}\label{eq:size_zt_bound}.
	\end{align}

	Now we define:
	\begin{align*}
        \theta &:= 	\frac{\sqrt{2\pi k}}{2}\binom{{\lceil 2\sqrt{m}k \rceil}}{k}^{-1/m}, &
	d&:=A(x^* - x').\\
	\theta'&:=2\sqrt{m}\theta.&
	d'&:=d- 1_m \theta'.
	\end{align*}
	Observe that
	\begin{align*}
        \theta = \frac{\sqrt{2\pi k}}{2}\binom{{\lceil2\sqrt{m}k\rceil}}{k}^{-1/m}\leq \frac{\sqrt{2\pi k}}{2} \left(2\sqrt{m}\right)^{-k/m}\leq \frac{1}{32m^2n}.
	\end{align*}
	So $\theta'\leq 1/8$.

If $|Z_t'|\geq \lceil 2\sqrt{m} \rceil kt$, then we can take $t$ disjoint
subsets $Z_t^{(1)},\ldots Z_t^{(k)}$ of $Z_t'$ of size $\lceil 2\sqrt{m}\rceil
k$. Conditioning on this event, we wish to apply \Cref{lemma:disc} to each
$\{B^{(t)}_i\}_{i\in Z_t^{(l)}}$, $l \in [t]$, to help us find a candidate
rounding of $x'$ to a ``good'' integer solution $x''$.

        Now we check that all conditions of \Cref{lemma:disc} are satisfied.
        By definition we have $\left(\frac{2\theta}{\sqrt{2\pi k}}\right)^m\binom{ak}{k} = 1$,
            and we can bound
	\begin{align*}
		\left\|(\Sigma^{(t)})^{-1} (Rd'- e_m k\mu_t )\right\|_2 &\leq \max(1,1/\sigma_t)(\|Rd\|_2 + \theta' +k\mu_t) \\
		&\leq \sqrt{10}(\|RA(x^*-x')\|_2 + \theta' + k\Delta t/2) \\&\leq \sqrt{10}\left(\sqrt{m}(4\sqrt{\log n} + \sqrt{m})/2+\frac{1}{8}+\frac{1}{4}\right)
		\\&\leq 4\sqrt{10m(\log n+m)}\leq  \sqrt{k}.
	\end{align*}

We now show that the
conditions of \Cref{lemma:disc} for $M=1$,$\gamma = 1/10$, and $k_0$ specified
below, are satisfied by $\{B^{(t)}_i\}_{i \in Z_t^{(l)}}$, $\forall l \in [t]$.

First, we observe that the $B^{(t)}_i$ are distributed as $(B^{(t)}_i)_m\sim
\sqrt{\epsilon}V+ \sqrt{1-\epsilon}U$ for $\epsilon =
\frac{1}{(1+\|u^*\|^2_2)\sigma_t^2}$, where $U$ is uniform on
$[-\sqrt{3},\sqrt{3}]$ and $V \sim \mathcal{N}(0,1)$. By
\Cref{lem:unif-gaus-conv}, $(B^{(t)}_i)_m$ is $(1/10, k_0)$-Gaussian convergent
for some $k_0$ and has maximum density at most $1$. Recalling that the
coordinates of $B^{(t)}_i$, $i \in Z_t'$, are independent and $(B^{(t)}_i)_j
\sim \mathcal{N}(0,1)$, $\forall j \in [m-1]$, we see that $B^{(t)}_i$ has
independent $(1/10,k_0)$-Gaussian convergent entries of maximum density at most $1$. Lastly, we note that
\begin{align*}
k &=165\,000m(\log(n) +m) \geq  165\,000(m^2 +k_0m ) \\&\geq \max \{ (4\sqrt{m} + 2)k_0, 144m^{\frac{3}{2}}(\log 1 + 3), 150\,000(\gamma + 1)m^{\frac{7}{4}} \}
\end{align*}
 as needed to apply \Cref{lemma:disc}, using that $n\geq \exp(k_0)$.

Therefore, applying \Cref{lemma:disc}, for each $l \in [t]$, with probability at least $1-1/25$, there exists a set $T_l \subseteq Z_t^{(l)}$ of size $k$ such that:
\begin{equation}
\label{eq:good-set}
	\left\|\sum_{i\in T_l}B_{i}^{(t)}-(\Sigma^{(t)})^{-1} (Rd'- e_m k\mu_t )\right\|_\infty\leq \theta.
\end{equation}

Call the event that \eqref{eq:good-set} is valid for any of the $t$ sets $E_t$.
	Because the success probabilities for each of the $t$ sets are independent, we get:
	\begin{align*}
	\Prob[\neg E_t\mid\ |Z_t'|\geq \lceil 2\sqrt{m} \rceil t k ]\leq \left(1-\frac{1}{25}\right)^{t}.
	\end{align*}
	Combining this with \Cref{eq:size_zt_bound}, we see that $\Prob[\neg E_t]\leq 2\cdot (1-\frac{1}{25})^{t}$. If $E_t$ occurs, we choose $T \subseteq Z_t'$, $|T|=k$, satisfying~\eqref{eq:good-set}. Then,
	\begin{align*}
	\left\|\sum_{i\in T}A_{\cdot,i}-d'\right\|_\infty&\leq
	\left\|\sum_{i\in T}A_{\cdot,i}-d'\right\|_2=\left\|\sum_{i\in T}B_{\cdot,i}-Rd'\right\|_2\\&=\left\|\sum_{i\in T}(\Sigma^{(t)}) B^{(t)}_{\cdot,i}+k\mu_te_m-Rd'\right\|_2\\
	&\leq \max(1,\sigma_t)\sqrt{m}\left\|\sum_{i\in T} B^{(t)}_{\cdot,i}-(\Sigma^{(t)})^{-1} (Rd'- e_m k\mu_t)\right\|_\infty\\
	&\leq 2 \sqrt{m}\theta  =\theta'.
	\end{align*}
	Now we will show that when $E_t$ occurs, $x''$ is feasible and $\|A(x''-x^*)\|_ \infty=O(1/n)$. First we check feasibility:
		\begin{align*}
	\sum_{i=1}^m x''_ia_{ji}&= (Ax' )_j + \sum_{i\in T}a_{ji} \leq (Ax' )_j+d'_j +   \theta'\\&=(Ax')_j+(A(x^*-x'))_j=(Ax^*)_j\leq b_j.
	\end{align*}
	Hence the solution is feasible for our LP.
	We also have \begin{align*}
\|A(x''-x^*)\|_\infty&=\|Ax''-Ax'-d\|_\infty\\&=\|\sum_{i\in T}A_{\cdot,i}-d'\|_\infty\leq \|\sum_{i\in T}A_{\cdot,i}-d\|_\infty +\theta'\leq 2\theta'.
	\end{align*}
        Now we can finalize our initial computation:
        \begin{align*}
	\val(x^*)-\val(x'')
	&\leq \sqrt{m}\|u^*\|_2\left\|A(x''-x^*)\right\|_\infty  + t\Delta k&\\
	&\leq 6\sqrt{m}\theta' +10\,000\cdot  \frac{\sqrt{m} \cdot  t\cdot k^2}{n}\\
	&\leq \frac{12}{32mn} + 27226\cdot 10^{10}  t\cdot \frac{m^{2.5}(\log n + m)^2}{n}\\
	&\leq    \finalHpUpperBoundConstant  t\cdot \frac{m^{2.5}(\log n + m)^2}{n}. \quad \qed
	\end{align*}
\end{proof}
Now we have all ingredients to prove \Cref{thm:ipgap_main_result}.
\begin{proof}[\Cref{thm:ipgap_main_result}]
	Substituting $\epsilon=1/9$ in \Cref{lem:primal-dual-props} gives
	\begin{align*}
\delta &=\sqrt{2\pi}\|b^-\|_2/n\leq \sqrt{2\pi}/10\leq 1/3\\
\alpha&=\frac{1}{\sqrt{2\pi}}\sqrt{\frac{1-3\epsilon}{1-\epsilon}-\delta^2}\geq \frac{1}{\sqrt{2\pi}}
	\sqrt{\left(\frac{3}{4}\right)^2-\delta^2}\\
	\|u^*\|_2&\leq \frac{\epsilon + 1}{1-3\epsilon-(1-\epsilon)\delta}\leq 3
	\end{align*}
	Now note that $H(499/500)< \frac{\alpha^2}{4} $, so $\beta < 499/500$. If we choose $n$, in such a way that $(499/500-\beta)n\geq m$, then $(1-\beta)n-m\geq n/500$. So \Cref{lem:primal-dual-props}	 yields that the probability that conditions 2 and 3 do not hold is at most $(20)^{m+1}\exp(-2^{-11}n)+\exp(-n/16\pi)$. We can still choose $C$ such that $n\geq Cm$ and  $t \leq \frac{n}{Cm^{2.5}(m+\log n)^2}$, so by taking $C\geq 2^{12}$, this probability is smaller than $2(1-1/25)^t$. By \Cref{lem:basis_bounded} condition 1 holds with probability at least $1-n^{-7}$.

	If 1, 2 and 3 hold  \Cref{lem:high_prob_set_T} implies that $\mathsf{IPGAP}(A,b,c) \geq \finalHpUpperBound$ holds with probability at most $2(1-1/25)^t$, as long as we set $C\geq \max(\exp(k_0), 10^{15})$ for $k_0$ as defined in \Cref{lem:unif-gaus-conv}. So:

	\begin{align*}
	&\Prob\left[\mathsf{IPGAP}(A,b,c) > \finalHpUpperBound\right]\\&\leq 2\cdot \left(1-\frac{1}{25}\right)^{t}
	 +n^{-7}+2\left(1-\frac{1}{25}\right)^{t}\\
	 &\leq 4(1-1/25)^t+n^{-7}. \quad \qed
	\end{align*}
\end{proof}

\section{Proof of the Improved Discrepancy Lemma}
\label{section:disc}

In this section, we prove \Cref{lemma:disc}, which is an improved variant of a discrepancy lemma of Dyer
and Frieze~\cite[Lemma 3.4]{dyer_probabilistic_1989}. This lemma is the main
tool used to restore feasibility and near-optimality of the rounding $x'$ of
$x^*$ in the preceding section. Compared to Dyer and Frieze's original lemma,
the main difference is that we achieve a constant probability of success
instead of $(2/\sqrt{3})^{-m}$, which is crucial for reducing the exponential
in $m$ dependence of the integrality gap down to polynomial in $m$. For this
purpose, we choose our subsets of size $k$ from a slightly larger universe,
i.e., $2 \sqrt{m}k$ instead of $2k$, to reduce correlations. We also provide a
tighter analysis of the lemma allowing us to show that the conclusion holds
as long as $k = \Omega(m^{7/4})$, where the $\Omega(\cdot)$ hides a dependence
on the parameters of the underlying coordinate distributions. We restate the lemma below.

\disclemma*

Our proof of \Cref{lemma:disc} follows the same proof strategy as~\cite[Lemma 3.4]{dyer_probabilistic_1989}. \\
Namely, we use the second moment method to lower bound the success
probability by $\E[Z]^2/\E[Z^2]$, where $Z$ counts the number of the
$k$-subsets of $[ak]$ which satisfy~\eqref{eq:disc-prob}. The value of
$\theta$ is calibrated to ensure that
\[
\E[Z] = \binom{ak}{k} \Prob\left[\length{(\sum_{j=1}^k Y_j)-D}_\infty \leq \theta \right] \approx 1.
\]
This relies upon the fact that $\sum_{j=1}^k Y_j$ is very close in
distribution to $\mathcal{N}(0,kI_m)$ and that the target $D$ is close to the
origin. To lower bound the ratio $\E[Z]^2/\E[Z^2]$, the key challenge from here is to
show that for two \emph{typical} $k$-subsets $K_1,K_2 \subset [ak]$, the
joint success probability is close to that of the \emph{independent} case,
i.e., where $K_1 \cap K_2 = \emptyset$. Namely, we wish to show that
\begin{equation}
\label{eq:cor-bnd}
\Prob \left[\length{(\sum_{j \in K_i}
Y_j)-D}_\infty \leq \theta, i \in \{1,2\}\right] \leq C
\Prob \left[\length{(\sum_{j=1}^k Y_j)-D}_\infty \leq \theta\right]^2
\end{equation}
for some not too large $C \geq 1$. This turns out to yield a lower bound
$\E[Z]^2/\E[Z]^2 = \Omega(1/C)$. To upper bound $C$, we rely upon the
following key technical lemma, which upper bounds the probability that two
correlated \emph{nearly-Gaussian} random sums land in the same interval.

\begin{lemma}
\label{lem:technical}
Let $r, k \in \mathbb{N}$ such that $\alpha := k/(k-r) \geq 4/3$. Let $Y_1, Y_2, \ldots Y_{k+r}$ be i.i.d. copies of a $(\gamma, k_0)$-Gaussian convergent random variable $Y$. Let $U =
\sum\limits_{j=1}^r Y_j, V = \sum\limits_{j=r+1}^k Y_j, W = \sum\limits_{j=k+1}^{r+k} Y_j$. Assume that $k \geq 280(\gamma+1) \alpha^{\frac{7}{2}}$, $\min \{r,k-r\} \geq k_0$ and $\theta \in \left[ 0,\frac{1}{\sqrt{k}} \right]$. For any $D \in \R$, we have that
$$\Prob[U+V \in [D -\theta, D +\theta], W+V \in [D -\theta, D +\theta]]
\leq  \frac{4\theta^2}{2\pi k}\left(1+\frac{16}{9 \alpha^2} \right).$$
\end{lemma}

The correlation is quantified by $\alpha > 0$, which controls the number of
shared terms $k/\alpha$ in both sums. As $\alpha$ increases, the probability
of the joint event approaches the worst-case bound for the independent
case. More precisely, if $X_1$ and $X_2$ are independent $\mathcal{N}(0,k)$
(i.e., sums of disjoint sets of $k$ independent standard Gaussians), $D
= 0$ and $\theta$ is very small, then
\begin{align*}
\Prob[X_1 \in [-\theta, \theta], X_2 \in [-\theta, \theta]]
&= \Prob_{X \sim \mathcal{N}(0,1)} \left[ X \in \left[-\frac{\theta}{\sqrt{k}},\frac{\theta}{\sqrt{k}}\right] \right]^2 \\ &\approx \left( \frac{2\theta}{\sqrt{k}} \varphi(0) \right)^2 = \frac{4\theta^2}{2\pi k}.
\end{align*}

In their proof, Dyer and Frieze proved a special case of the above for
$\alpha \approx 2$. In their setting, they pick $K_1,K_2 \subseteq [2k]$ of
size $k$, where the \emph{typical} intersection size is $|K_1 \cap K_2| \approx
k/2$. To upper bound $C$ in~\eqref{eq:cor-bnd}, they
apply~\Cref{lem:technical} to each of the coordinates $\sum_{j \in K_1}
Y_{j,i}, \sum_{j \in K_2} Y_{j,i}$, for $i \in [m]$, with $\alpha \approx 2$,
to deduce a $2^{O(m)}$ bound on $C$. By increasing the universe size from
$[2k]$ to $[ak]$, where $a = \lceil{2\sqrt{m}\rceil}$, we reduce the typical
intersection size to $k/a$. This allows us to set $\alpha \approx 1/a$, which reduces $C$ to roughly $(1+\frac{1}{a^2})^m = O(1)$.

\paragraph{\bf Remark on Gaussian convergence.} The attentive reader may have
noticed that our definition of $(\gamma, k_0)-$Gaussian convergence requires
the density of the normalized sum of $k$ i.i.d.~copies of a random variable $X
\in \R$ to be within $\frac{\gamma}{k}$ of the density of the standard
Gaussian, for $k \geq k_0$, which is a stronger requirement than the more
conventional $\frac{\gamma}{\sqrt{k}}$-type convergence. Indeed, this faster
rate of convergence can only be expected when the first three moments of $X$
match those of the standard Gaussian (as is the case in our application),
whereas the slower rate requires only that $X$ have mean $0$ and variance $1$
and that the density of $X$ be sufficiently ``nice''. One can show however
that \Cref{lemma:disc} still holds only assuming
$\frac{\gamma}{\sqrt{k}}$-convergence, provided that the requirement $k \geq
150000(\gamma+1)m^{7/4}$ is replaced by $k \geq 
150000^2 (\gamma+1)^2 m^{7/2}$. It is in fact easy to verify that
every inequality in the proof of \Cref{lemma:disc} utilizing Gaussian
convergence is preserved when replacing $k$ by $\sqrt{k}$, and hence the
required lower bound on $k$ is replaced by the same lower bound for
$\sqrt{k}$. The larger $m^{7/2}$ dependence on $m$ would however increase the
dependence on $m$ in~\Cref{thm:ipgap_main_result}, which is why we chose to
state~\Cref{lemma:disc} with the stronger Gaussian convergence requirement. 
\vspace{1em}

We now prove \Cref{lemma:disc} using \Cref{lem:technical}, deferring the
proof of the latter to the end of the section.

\begin{proof} [\Cref{lemma:disc}]
For every $K \subset [ak]$ let $E_K$ denote the event that $\length{\sum\limits_{j \in K} Y_j - D}_\infty \leq \theta$. Let $Z = \sum_{K \subset [ak], |K|=k} \mathbf{1}_{E_K}$.
$$\Prob \left[ \exists K \subset [ak] : |K| = k, \length{\sum\limits_{j\in K} Y_{j}-D}_\infty \leq \theta \right] = \Prob[Z>0].$$

By an application of the second moment method, if $Z$ is a positive-integer valued random variable of finite mean and variance then

$$\Prob[Z > 0] \geq \frac{\E[Z]^2}{\E[Z^2]}.$$

Define $K_i = \set{i+1, i+2, \dots, i+k}, 0 \leq i \leq k.$
\begin{align*}
\E[Z] &= \sum_{K \subset [ak], |K| = k} \Prob[E_K] = {ak \choose k} \Prob[E_{K_0}], \\
\E[Z^2] &= \sum_{K, K' \subset [ak], |K|=|K'|=k} \Prob[E_{K} \cap E_{K'}],
\end{align*}
where the first equality follows since $Y_1,\dots,Y_{ak}$ are independent and identically distributed. By the same reasoning, $\Prob[E_{K} \cap E_{K'}]$ depends only on $|K \cap K'|$. Noting that $\lenfit{K_0 \cap K_r} = k-r$ for $0 \leq r \leq k$, we may rewrite $\E[Z^2]$ as
$$\E[Z^2] = {ak \choose k}^2 \sum\limits_{r=0}^k \Prob[\lenfit{K \cap K'} = k-r] \Prob[E_{K_0} \cap E_{K_r}],$$
where $K,K'$ are independent uniformly random subsets of $[ak]$ of size $k$.
Since we have $\Prob[\lenfit{K \cap K'} = k] = {ak \choose k}^{-1},$
$$\E[Z^2] = \E[Z] + {ak \choose k}^2 \sum\limits_{r=1}^k \Prob[\lenfit{K \cap K'} = k-r] \Prob[E_{K_0} \cap E_{K_r}].$$

In the rest of the proof, we show $\E[Z^2]/\E[Z]^2 \leq 25$, which implies the desired lower bound $\E[Z]^2/\E[Z^2] \geq 1/25$. This will be established for sufficiently large $k$ and the required constraints for $k$ will be collected at the end. Let
$$J := [k] \cap \left( k - \frac{(1+\varepsilon)k}{a}, k - \frac{(1-\varepsilon)k}{a} \right) \text{ for } \varepsilon = \frac{1}{2}.$$
We decompose
\begin{align}\E[Z^2]/\E[Z]^2 = \frac{1}{\E[Z]} &+ \sum_{r \in J} \Prob[|K \cap K'|=k-r]\frac{\Prob[E_{K_0} \cap E_{K_r}]}{\Prob[E_{K_0}]^2} \nonumber \\ &+ \sum_{r \in [k] \backslash J} \Prob[|K \cap K'|=k-r]\frac{\Prob[E_{K_0} \cap E_{K_r}]}{\Prob[E_{K_0}]^2}. \label{eq:ratio-decomp}
\end{align}

We upper bound $\E[Z^2]/\E[Z]^2$ in three steps. First, a lower bound on $\E[Z]$. Then, a worst-case upper bound on $\frac{\Prob[E_{K_0} \cap E_{K_r}]}{\Prob[E_{K_0}]^2}$ for every $1 \leq r \leq k$, and finally a superior average-case upper bound on $\frac{\Prob[E_{K_0} \cap E_{K_r}]}{\Prob[E_{K_0}]^2}$ for $r \in J$.

Let us prove a trivial upper bound for $\theta$. Since ${ak \choose k} \geq a^k$, we have
$$\left(\frac{2\theta}{\sqrt{2\pi k}} \right)^m {ak \choose k} = 1 \Rightarrow \theta \leq a^{-k/m} \sqrt{\frac{\pi k}{2}}.$$
If $k \geq 2m \log m $, then $k / m \geq \log k$, so that
$$\theta \leq \left( 2 \sqrt{m} \right)^{-\frac{k}{m}} \sqrt{\frac{\pi k}{2}} \leq 2^{-\frac{k}{m}} \sqrt{ \frac{\pi k} {2}} \leq \frac{1}{\sqrt{k}}.$$
Henceforth, we will assume $\theta \leq 1/\sqrt{k}$ and require $k \geq 2m \log m$.

\paragraph{\bf Lower Bound for the Expectation.} The goal of this subsection is to show that

$$\E[Z] = {ak \choose k} \Prob[E_{K_0}] \geq 1/e.$$

Since the individual coordinates of each vector $Y_j$, $j \in [k]$, are
independent:

$$\Prob[E_{K_0}] = \prod\limits_{i=1}^m \Prob\left[ \sum\limits_{j=1}^k Y_{j, i} \in [D_i - \theta, D_i + \theta]\right].$$

Each term in this product is at least
$$
\Prob\left[ \sum\limits_{j=1}^k Y_{j, i} \in [D_i - \theta,
D_i + \theta]\right] \geq \frac{2\theta}{\sqrt{2\pi k}} \exp
\left(-\frac{D_i^2}{2k} - \frac{2e\sqrt{2\pi}(\gamma+1)}{k}\right) , \forall
1 \leq i \leq m.$$
To prove this estimate, let $\bar{g}_i^{k}$ denote the density function of
$\sum\limits_{j=1}^k Y_{j, i} /\sqrt{k}$. Then,

\[
\Prob\left[ \sum_{j=1}^k Y_{j, i} \in [D_i - \theta, D_i + \theta]\right] \hspace{-.2em} = \Prob\left[ \sum_{j=1}^k \frac{ Y_{j, i}}{\sqrt{k}} \in \left[ \frac{D_i - \theta}{\sqrt{k}}, \frac{D_i + \theta}{\sqrt{k}} \right] \right]
\hspace{-.2em} = \int_{\frac{D_i - \theta}{\sqrt{k}}} ^{\frac{D_i + \theta}{\sqrt{k}}} \bar{g}_i^{k}(x) \mathrm{d} x.
\]

Recall that for $i \in [m]$, $Y_{1,i},\dots,Y_{k,i}$ are i.i.d. $(\gamma, k_0)$-Gaussian convergent random variables. If $k \geq k_0$, we have that
$$\bar{g}_i^{k}\left(x \right) \geq \varphi \left( x \right) - \frac{\gamma}{k}, \forall x \in \R,$$
where $\varphi(x) := \frac{\exp(-x^2/2)}{\sqrt{2\pi}}$. Assuming that $k \geq 4\sqrt{2\pi} e (\gamma +1)m$, the desired estimate is derived as follows:
\begin{align*}
\int_{\frac{D_i - \theta}{\sqrt{k}}}^{\frac{D_i + \theta}{\sqrt{k}}} \bar{g}_i^{k}(x) \mathrm{d} x
 &\geq \int_{\frac{D_i - \theta}{\sqrt{k}}}^{\frac{D_i + \theta}{\sqrt{k}}} \varphi \left( x \right) - \frac{\gamma}{k} \mathrm{d} x \\
&\geq \int_{\frac{D_i - \theta}{\sqrt{k}}}^{\frac{D_i + \theta}{\sqrt{k}}} \varphi \left( \frac{D_i}{\sqrt{k}} \right) - \frac{\theta}{\sqrt{k}} - \frac{\gamma}{k} \mathrm{d} x \quad \left(~ \varphi \text{ is 1-Lipschitz } \right)\\
&\geq \frac{2\theta}{\sqrt{2\pi k}} \left( \exp \left(-\frac{D_i^2}{2k} \right) - \frac{\sqrt{2\pi} \theta}{\sqrt{k}} - \frac{\sqrt{2\pi}\gamma}{k}\right) \\
&\geq \frac{2\theta}{\sqrt{2\pi k}} \left( \exp \left(-\frac{D_i^2}{2k} \right) - \frac{\sqrt{2\pi}(\gamma+1)}{k}\right) \quad \left(\text{ since } \theta \leq \frac{1}{\sqrt{k}} \right) \\
&\geq \frac{2\theta}{\sqrt{2\pi k}} \exp \left(-\frac{D_i^2}{2k} \right)\left( 1 - \frac{2\sqrt{2\pi}(\gamma+1)}{k} \right)\\& \hspace{4cm} \left(~\length{D}_2 \leq \sqrt{k} \Rightarrow \exp \left(-\frac{D_i^2}{2k}\right) \geq 1/2 \right) \\
&\geq \frac{2\theta}{\sqrt{2\pi k}} \exp \left( -\frac{D_i^2}{2k} - \frac{2e\sqrt{2\pi}(\gamma+1)}{k} \right) \\&\hspace{4cm} \left( \text{ since } 1-x/e \geq e^{-x}, 0 \leq x \leq 1~ \right).
\end{align*}

Building on this estimate, we now consider all coordinates and get
\begin{align*}
\Prob[E_{K_0}] &\geq \left( \frac{2\theta}{\sqrt{2\pi k}} \right)^m \prod\limits_{i=1}^m \exp \left( -\frac{D_i^2}{2k} - \frac{2e\sqrt{2\pi}(\gamma+1)}{k} \right)  \\
&\geq \left( \frac{2\theta}{\sqrt{2\pi k} }\right)^m \exp \left( -\sum\limits_{i=1}^m \frac{D_i^2}{2k} - \frac{1}{2} \right) \quad \left(\text{ since } k \geq 4\sqrt{2\pi} e (\gamma +1)m~\right)\\
&\geq \left( \frac{2\theta}{\sqrt{2\pi k}} \right)^m \exp \left( -1\right) \quad \left( \text{ since } \length{D_i}_2 \leq \sqrt{k} ~ \right).
\end{align*}

Since $\left(\frac{2\theta}{\sqrt{2\pi k}}\right)^m {ak \choose k} = 1,$ we conclude that $\E[Z] \geq 1/e$, as needed.

\paragraph{\bf Upper bound on $\frac{\Prob[E_{K_0} \cap E_{K_r}]}{\Prob[E_{K_0}]^2}$.}

To analyze $\Prob[E_{K_0} \cap E_{K_r}]$, for $r \in [k]$, we define
$$U_i := \sum\limits_{j=1}^r Y_{j,i}, V_i := \sum\limits_{j=r+1}^k Y_{j,i}, W_i := \sum\limits_{j=k+1}^{r+k} Y_{j,i}, \text{ for } i \in [m].$$
The probability of $E_{K_0} \cap E_{K_r}$ can now be expressed as follows:

\begin{equation}
\label{eq:cor-prob}
\Prob[E_{K_0} \cap E_{K_r}] = \prod\limits_{i=1}^m \Prob[U_i + V_i \in [D_i -\theta, D_i+\theta], W_i+V_i \in [D_i -\theta, D_i+\theta]].
\end{equation}

For $i \in [m]$, let $g_i: \R \rightarrow \R_+$ denote the probability density of $Y_{1,i}$, which satisfies $\sup_{x \in \R} g_i(x) \leq M$ by assumption. Since $Y_{1,i},\dots,Y_{k,i}$ are i.i.d., we
see that $U_i,W_i,V_i$ are independent, that $U_i,W_i$ both have density
$g_i^{*r}$ and that $V_i$ has density $g_i^{*(k-r)}$, where
$g_i^{*r},g_i^{*(k-r)}$ are the $r$ and $k-r$-fold convolutions of $g_i$.

We now split the analysis into a worst-case bound for any $r \in [k]$ and an average case bound for $r \in J$.

\paragraph{\bf Worst-case Upper Bound.} Since the convolution operation does not increase the maximum density, we note that $\max_{x \in \R} g_i^{*r}(x) \leq M$, $i \in [m]$. From here, we see that
\begin{align*}
\Prob[U_i + V_i &\in [D_i-\theta, D_i+\theta], W_i+V_i \in [D_i -\theta, D_i+\theta]] \\
&= \int_{-\infty}^\infty \Prob[U_i \in [D_i-\theta-y, D_i+\theta-y]]^2 g_i^{*(k-r)}(y) d{\rm y} \\
&= \int_{-\infty}^\infty \left(\int_{-\theta}^{\theta} g^{*r}_i(D_i-y+x) d{\rm x}\right)^2 g_i^{*(k-r)}(y) d{\rm y} \\
&\leq \int_{-\infty}^\infty \left(\int_{-\theta}^{\theta} M d{\rm x}\right)^2 g_i^{*(k-r)}(y) d{\rm y} = (2M\theta)^2.
\end{align*}
By~\eqref{eq:cor-prob}, this gives $\Prob[E_{K_0} \cap E_{K_r}] \leq (2\theta M)^{2m}$ for $r \in [k]$. The worst case upper bound is therefore
$$\frac{\Prob[E_{K_0} \cap E_{K_r}]}{\Prob[E_{K_0}]^2} \leq e^2 \left( \frac{(2\theta M) \sqrt{2\pi k}}{2\theta } \right)^{2m} \leq e^2 \left( M\sqrt{2\pi k} \right)^{2m}.$$

Before moving on to the average case bound for $\Prob[E_{K_0} \cap E_{K_r}]^2/\Prob[E_{K_0}]^2$ for $r \in J$, we first upper bound the total contribution to~\eqref{eq:ratio-decomp} of the terms associated with $r \in [k] \setminus J$.  Recall that $J = [k] \cap \left( k-\frac{(1+\varepsilon)k}{a}, k - \frac{(1-\varepsilon)k}{a} \right)$ for $\varepsilon = \frac{1}{2}$. For $r \in J$, note that $|K_0 \cap K_r| \in (\frac{k}{2a},\frac{3k}{2a})$. That is, $|K_0 \cap K_r|$ is close to the average intersection size $\frac{k}{a}$ for two uniform $k$-subsets $K,K'$ of $[ak]$. Applying \Cref{lem:intersection} together with the worst-case upper bound for indices $r \in [k] \setminus J$, assuming that $k$ is large enough (to be specified below), we get that

\begin{align}
\sum\limits_{r \in [k] \setminus J} \Prob[\lenfit{K \cap K'} = k-r] &\frac{\Prob[E_{K_0} \cap E_{K_r}]}{\Prob[E_{K_0}]^2} \nonumber \\ &\leq e^2 (M\sqrt{2\pi k})^{2m} \Prob \left[\len{K \cap K'} \not\in (\frac{k}{2a},\frac{3k}{2a})\right] \nonumber \\
& \leq 4e^2 (M\sqrt{2\pi k})^{2m} \exp(-\frac{k}{12a}) \leq 1 \label{eq:worst-case-bnd}.
\end{align}
To establish the last inequality, using that $a = \ceil{2\sqrt{m}} \leq
3\sqrt{m}$ and $x/c \geq \log x$ for $x \geq 2c \log c$ for $c \geq 1$, it suffices to show that the logarithm of the last expression is non-positive:
\begin{align*}
&\phantom{\leq}\log(4e^2) + 2m \left(\log M + \frac{\log (2\pi k)}{2} \right) - \frac{k}{12a} \\&\leq m \left( \log(8e^2\pi) + 2\log M + \log k - \frac{k}{36m^{3/2}} \right) \\
           &\leq m \left( \log(8e^2\pi) + 2\log M - \frac{k}{72m^{3/2}} \right) \quad \left( \text{ if } k \geq 144m^{3/2}\log (72 m^{3/2}) \right), \\
           &\leq 0 \quad \left( \text{ if } k \geq 72m^{3/2}(2\log M
+ \log(8e^2\pi)) \right).
\end{align*}
Simplifying the conditions above,~\eqref{eq:worst-case-bnd} holds for
$$ k \geq \max \{144m^{3/2}(\log M + 3), 216m^{3/2}(\log m + 3)\}.$$

\paragraph{\bf Average-case Upper Bound.}

Observe that for $r \in J$, we have $r = k - \frac{k}{\alpha}$ for $\alpha
\in \left[ \frac{4}{3}\sqrt{m}, 4\sqrt{m} + 2\right]$. To derive the average
case bound, we will apply \Cref{lem:technical} to $U_i,W_i,V_i, D_i$ and $\theta$ for each $i
\in [m]$ with parameters $k$,$r$,$\alpha$ and $\gamma$, $k_0$. We first show
that the requisite conditions are satisfied for $k$ large enough. Firstly,
recall that $\theta \in [0,1/\sqrt{k}]$, $\alpha \geq 4/3\sqrt{m} \geq 4/3$, and that $Y_{1,i},\dots,Y_{k,i}$ are i.i.d. $(\gamma,k_0)$-Gaussian
convergent random variables. Assuming $k \geq (4\sqrt{m}+2)k_0$, we have that $\min
\{r,k-r\} = k \min \{\frac{1}{\alpha},(1-\frac{1}{\alpha})\} \geq k_0$. Lastly,
assuming $k \geq 280(\gamma+1)(4\sqrt{m}+2)^{7/2}$, we also have $k \geq
280(\gamma+1)\alpha^{7/2}$ by assumption on $\alpha$. Therefore, for $r \in J$, we may apply
\Cref{lem:technical} to~\eqref{eq:cor-prob} to conclude that

\begin{align*}
\frac{\Prob[E_{K_0} \cap E_{K_r}]}{\Prob[E_{K_0}]^2} &\leq \left( \left( 1 + \frac{16}{9\alpha^2} \right) \frac{4\theta^2}{2\pi k}\right)^m e^2\left(\frac{\sqrt{2\pi k}}{2\theta} \right)^{2m} \leq e^2 \left(1 + \frac{16}{9\alpha^2}\right)^m \\
 &\leq e^2 \left(1 + \frac{1}{m}\right)^m \leq e^3.
\end{align*}
Because $4\sqrt{m}+2 \leq 6 \sqrt{m}$, we simplify the condition
$k \geq 280(\gamma+1)(4\sqrt{m}+2)^{7/2}$ to the stronger condition $k \geq
150000(\gamma + 1)m^{7/4}$.

\paragraph{\bf In Conclusion:}

Combining the bounds from the previous sections, assuming that the constraints on $k$ are all satisfied, we get that:

\begin{align*}
\E[Z^2]/\E[Z]^2 &= \frac{1}{\E[Z]} + \sum_{r=1}^k \Prob[|K \cap K'|=k-r]\frac{\Prob[E_{K_0} \cap E_{K_r}]}{\Prob[E_{K_0}]^2} \\
&\leq e + \sum_{r \in [k] \setminus J} \Prob[|K \cap K'|=k-r]\frac{\Prob[E_{K_0} \cap E_{K_r}]}{\Prob[E_{K_0}]^2} \\
& \quad \quad \quad \quad  + \sum_{r \in J} \Prob[|K \cap K'|=k-r]\frac{\Prob[E_{K_0} \cap E_{K_r}]}{\Prob[E_{K_0}]^2}\\
&\leq e + \sum_{r \in [k] \backslash J} \Prob[|K \cap K'|=k-r]e^2\left(M \sqrt{2\pi k} \right)^m \\
&\quad \quad \quad \quad + \sum_{r \in J} \Prob[|K \cap K'|=k-r]e^3 \\
&\leq e + 1 +  e^3 \leq 25.
\end{align*}

Aggregating the constraints on $k$, the above gives a lower bound of $1/25$ for the success probability whenever $k$ satisfies:
\begin{align*}
k \geq \max \{ 2m \log m, k_0, &~4\sqrt{2 \pi} e (\gamma +1) m,~144m^{3/2} (\log M + 3),~216m^{3/2}(\log m + 3), \\ & (4\sqrt{m}+2)k_0,~150000(\gamma + 1)m^{7/4} \}.
\end{align*}

Removing the dominated terms, it suffices for $k$ to satisfy
$$
k \geq \max \{ (4\sqrt{m} + 2)k_0,~144 m^{\frac{3}{2}}(\log M + 3),~150000(\gamma + 1)m^{\frac{7}{4}} \},$$
as needed. \qed
\end{proof}

\begin{proof}[\Cref{lem:technical}]
For convenience, let $P = \Prob[U + V \in [D -\theta, D +\theta], W+V \in [D -\theta, D +\theta]]$, and define the normalized sums: $\bar{U} = U/\sqrt{r}, \bar{V} = V/\sqrt{k-r}, \bar{W} = W/\sqrt{r}$. Let us denote the density of $\bar{U}$ and $\bar{W}$ by $\bar{g}$, and the density of $\bar{V}$ by $\bar{h}$. $\bar{U}, \bar{W}$ are independent and identically distributed, from which we see that
\begin{align*}
P &= \Prob \left[\bar{U}\sqrt{r} + \bar{V}\sqrt{k-r} \in [D -\theta, D+\theta], \bar{W} \sqrt{r} + \bar{V}\sqrt{k-r} +  \in [D -\theta, D +\theta] \right]\\
&= \int\limits_{-\infty}^\infty \Prob \left[\bar{U} \in \left[ \frac{D - \theta}{\sqrt{r}} - y\sqrt{\frac{k-r}{r}}, \frac{D + \theta}{\sqrt{r}} - y\sqrt{\frac{k-r}{r}} \right] \right]^2 \bar{h}(y) \\
&= \int\limits_{-\infty}^\infty  \left( \int\limits_{-\frac{\theta}{\sqrt{r}}}^{\frac{\theta}{\sqrt{r}}} \bar{g}\left(\frac{D}{\sqrt{r}} - \frac{y}{\sqrt{\alpha - 1}} + x \right) \mathrm{d}x \right)^2 \bar{h}(y).
\end{align*}
 Since $Y$ is $(\gamma, k_0)$-Gaussian convergent and $\min \{r,k-r\} \geq k_0$, we have that
$$\len{\bar{g}(x) - \phi(x)} \leq \frac{\gamma}{r} = \frac{\alpha}{\alpha-1}\frac{\gamma}{k}, \quad \len{\bar{h}(x) - \phi(x)} \leq \frac{\gamma}{k-r} = \alpha\frac{\gamma}{k}, \quad \forall x \in \R.$$
Using the above, we upper bound $P$ as follows:
\begin{align*}
P &= \int\limits_{-\infty}^\infty \left( \int\limits_{-\frac{\theta}{\sqrt{r}}}^{\frac{\theta}{\sqrt{r}}} \bar{g}\left(\frac{D}{\sqrt{r}} - \frac{y}{\sqrt{\alpha - 1}} + x \right) \mathrm{d}x \right)^2 \bar{h}(y) \mathrm{d}y\\
&\leq \int\limits_{-\infty}^\infty \left( \int\limits_{-\frac{\theta}{\sqrt{r}}}^{\frac{\theta}{\sqrt{r}}} \varphi \left(\frac{D}{\sqrt{r}} - \frac{y}{\sqrt{\alpha - 1}} + x \right) + \frac{\gamma}{k}\frac{\alpha}{\alpha-1} \mathrm{d}x \right)^2 \bar{h}(y) \mathrm{d}y\\
&\leq \int\limits_{-\infty}^\infty \left( \int\limits_{-\frac{\theta}{\sqrt{r}}}^{\frac{\theta}{\sqrt{r}}} \varphi \left(\frac{D}{\sqrt{r}} - \frac{y}{\sqrt{\alpha - 1}} \right) + \frac{2\gamma}{k} + \frac{\theta}{\sqrt{r}} \mathrm{d}x \right)^2 \bar{h}(y) \mathrm{d}y \quad \left(~\varphi \text{ is 1-Lipschitz } \right)\\
&\leq \frac{4\theta^2}{r} \int\limits_{-\infty}^\infty \left( \varphi \left(\frac{D}{\sqrt{r}} - \frac{y}{\sqrt{\alpha - 1}} \right) + \frac{2\gamma}{k} + \frac{\theta}{\sqrt{r}} \right)^2 \bar{h}(y) \mathrm{d}y\\
&\leq \frac{4\theta^2}{r} \left( 2\varphi(0)\left( \frac{2\gamma}{k} + \frac{\theta}{\sqrt{r}} \right) + \left( \frac{2\gamma}{k} + \frac{\theta}{\sqrt{r}} \right)^2 + \int\limits_{-\infty}^\infty \varphi \left(\frac{D}{\sqrt{r}} - \frac{y}{\sqrt{\alpha - 1}} \right)^2 \bar{h}(y) \mathrm{d}y \right)\\
&\leq \frac{4\theta^2}{r} \left( \frac{2(\gamma + 1)}{k} + \left( \frac{2(\gamma + 1)}{k} \right)^2 + \int\limits_{-\infty}^\infty \varphi \left(\frac{D}{\sqrt{r}} - \frac{y}{\sqrt{\alpha - 1}} \right)^2 \bar{h}(y) \mathrm{d}y \right),
\end{align*}
where the last inequality follows from $\theta \leq \frac{1}{\sqrt{k}}$, $r = (1-\frac{1}{\alpha})k \geq k/4$ for $\alpha \geq 4/3$ and $\varphi(0) \leq \frac{1}{2}$. Next, we upper bound the term $Q = \int\limits_{-\infty}^\infty \varphi \left(\frac{D}{\sqrt{r}} - \frac{y}{\sqrt{\alpha - 1}} \right)^2 \bar{h}(y) \mathrm{d}y$.
\begin{align*}
Q &= \int\limits_{-\infty}^\infty \varphi \left(\frac{D}{\sqrt{r}} - \frac{y}{\sqrt{\alpha - 1}} \right)^2 \bar{h}(y) \mathrm{d}y\\
&\leq \int\limits_{-\infty}^\infty \varphi \left(\frac{D}{\sqrt{r}} - \frac{y}{\sqrt{\alpha - 1}} \right)^2 \left(\varphi(y) + \frac{\alpha\gamma}{k} \right) \mathrm{d}y \\
&= \frac{\alpha \gamma}{k} \int\limits_{-\infty}^\infty \varphi \left(\frac{D}{\sqrt{r}} - \frac{y}{\sqrt{\alpha - 1}} \right)^2 \mathrm{d}y + \int\limits_{-\infty}^\infty \varphi \left(\frac{D}{\sqrt{r}} - \frac{y}{\sqrt{\alpha - 1}} \right)^2 \varphi(y) \mathrm{d}y\\
&= \frac{\alpha \sqrt{\alpha -1}\gamma}{\sqrt{2} k} \int\limits_{-\infty}^\infty \varphi \left(\frac{y}{\sqrt{2}} \right)^2 \mathrm{d}y + \int\limits_{-\infty}^\infty \varphi \left(\frac{D}{\sqrt{r}} - \frac{y}{\sqrt{\alpha - 1}} \right)^2 \varphi(y) \mathrm{d}y \\
&= \frac{\alpha \sqrt{\alpha -1}\gamma}{2\sqrt{\pi} k}  + \frac{1}{(2\pi)^{3/2}}\int\limits_{-\infty}^\infty \exp \left( -\left(\frac{D}{\sqrt{r}} - \frac{y}{\sqrt{\alpha - 1}} \right)^2 - \frac{y^2}{2} \right) \mathrm{d}y \\
&= \frac{\alpha \sqrt{\alpha -1}\gamma}{2\sqrt{\pi} k}  + \frac{1}{(2\pi)^{3/2}}\int\limits_{-\infty}^\infty \exp \left( -\frac{y^2}{2}\left(1 + \frac{2}{\alpha - 1} \right) + \frac{2yD}{\sqrt{r(\alpha - 1)}} - \frac{D^2}{r} \right) \mathrm{d}y \\
&\leq \frac{\alpha \sqrt{\alpha -1}\gamma}{2\sqrt{\pi} k} \\& + \frac{1}{(2\pi)^{3/2}}\int\limits_{-\infty}^\infty \exp \left( -\frac{1}{2}\left(y\sqrt{\frac{\alpha+1}{\alpha -1}} - \frac{2 D}{\sqrt{(\alpha +1)r}} \right)^2 + \frac{2D^2}{(\alpha+1)r} - \frac{D^2}{r} \right) \mathrm{d}y \\
&\leq \frac{\alpha \sqrt{\alpha -1}\gamma}{2\sqrt{\pi} k}  + \frac{1}{(2\pi)^{3/2}}\int\limits_{-\infty}^\infty \exp \left( -\frac{1}{2} \left(y \sqrt{\frac{\alpha+1}{\alpha-1}}\right)^2 \right) \mathrm{d}y  \\&\quad \left(\text{because } \alpha \geq \frac{4}{3} \Rightarrow \frac{2D^2}{(\alpha+1)r} - \frac{D^2}{r} \leq 0 \right)\\
&= \frac{\alpha \sqrt{\alpha -1}\gamma}{2\sqrt{\pi} k} + \frac{1}{2\pi}\sqrt{\frac{\alpha - 1}{\alpha+1}}.
\end{align*}

The final expression is:
$$P \leq \frac{4\theta^2}{r} \left(  \frac{2(\gamma + 1)}{k} + \left( \frac{2(\gamma + 1)}{k} \right)^2 + \frac{\alpha \sqrt{\alpha -1}\gamma}{2\sqrt{\pi} k} + \frac{1}{2\pi}\sqrt{\frac{\alpha - 1}{\alpha+1}}\right).$$
Since $r = \frac{\alpha -1}{\alpha} k$, we have
$$P \leq \frac{4\theta^2}{ k} \left( 1 + \frac{1}{\alpha - 1} \right) \left(  \frac{2(\gamma + 1)}{k} + \left( \frac{2(\gamma + 1)}{k} \right)^2 + \frac{\alpha \sqrt{(\alpha -1)}\gamma}{2\sqrt{\pi}k} + \frac{1}{2\pi}\sqrt{\frac{\alpha - 1}{\alpha+1}}\right).$$

We require $P \leq \frac{4\theta^2}{2\pi k} \left( 1 + \frac{16}{9\alpha^2} \right)$. Using $\alpha \geq 4/3$, observe that
\begin{align*}
\left( 1 + \frac{1}{\alpha - 1} \right) \sqrt{\frac{\alpha - 1}{\alpha+1}} = \sqrt{1+\frac{1}{\alpha^2 - 1}} &\leq 1 + \frac{1}{2(\alpha^2 - 1)}  \\&´\leq 1 + \frac{16}{14\alpha^2} \quad (\text{ since } 14 \alpha^2 \leq 32(\alpha^2 - 1)~).\end{align*}
So all we need is
$$ \left(1 + \frac{1}{\alpha - 1} \right) \left(  \frac{2(\gamma + 1)}{k} + \left( \frac{2(\gamma + 1)}{k} \right)^2 +\frac{\alpha \sqrt{(\alpha -1)}\gamma}{2\sqrt{\pi}k} \right) \leq \frac{4}{(2\pi)7 \alpha^2},$$
because $\frac{16}{14} + \frac{8}{14} \leq \frac{16}{9}$. Using $\alpha \geq \frac{4}{3}$, we simplify the condition to
$$ \frac{2(\gamma + 1)}{k} + \left( \frac{2(\gamma + 1)}{k} \right)^2 +\frac{\alpha \sqrt{(\alpha -1)}\gamma}{2\sqrt{\pi}k} \leq \frac{1}{(2\pi)7 \alpha^2}.$$
A stronger condition is $3 \cdot \frac{\alpha^{3/2} 2(\gamma+1)}{k} \leq
\frac{1}{(2\pi)7\alpha^2}$, which is satisfied whenever $k \geq 280 (\gamma +1) \alpha^{\frac{7}{2}}$. \qed
\end{proof}

 \section{Bounding the Tree Size in Branch-and-Bound}
 \label{sec:bnb-adapt}
 
The goal of this section is to prove~\Cref{thm:meta-logcon} from the
introduction. Our main technical lemma, which will allow us to upper
bound~\eqref{eq:bnb-size}, is given below. 

 \begin{lemma}
 \label{lem:knap-bnd}
Let $d \in \N, n \geq 100 d$, $G \geq 0$, and $W = (W_1,\dots,W_n) \in \R^{d \times n}$ be a
matrix whose columns are independent logconcave random vectors with
identity covariance. Then, for $\delta \in (0,1)$, with
probability at least $1-\delta-e^{-n/5}$, we have that
 \[
 \max_{\|u\|_2=1} |\{x \in \{0,1\}^n: \sum_{i=1}^n x_i |u^\T W_i| \leq G\}| \leq 60 (357 n^6)^{d+1} e^{2\sqrt{2nG}}/\delta. 
 \] 
 \end{lemma}

Before proving~\Cref{lem:knap-bnd}, we first show how to
derive~\Cref{thm:meta-logcon} from~\Cref{thm:bnb} and~\Cref{lem:knap-bnd}.

\begin{proof}[of~\Cref{thm:meta-logcon}]
 For any $\lambda \in \R^m$, we see that
 \begin{align*}
 \{x \in
 \{0,1\}^n: \sum_{i=1}^n x_i |(c-A^\T \lambda)_i| \leq G\}
 &= \{x \in \{0,1\}^n: \sum_{i=1}^n x_i |(W^\T(1,-\lambda))_i| \leq tG\} \\
 &\subseteq \{x \in \{0,1\}^n: \sum_{i=1}^n x_i |(\frac{W^\T(1,-\lambda)}{\|(1,-\lambda)\|_2} )_i| \leq G\}.
 \end{align*}
 Given the above, we have that
 \begin{align}
 \max_{\lambda \in \R^m} |\{x \in
 \{0,1\}^n:& \sum_{i=1}^n x_i |(c-A^\T \lambda)_i| \leq G\}| \nonumber \\
 &\leq 
 \max_{\|u\|_2=1} |\{x \in
 \{0,1\}^n: \sum_{i=1}^n x_i |( W^\T u)_i| \leq G\}|.
 \label{eq:gb-ub-1}
 \end{align}
 
 Applying \Cref{lem:knap-bnd} to $W$, $\delta$ and $d = m+1$,
 with probability at least $1-\delta-e^{-n/5}$, we get that 
 \begin{align}
 \max_{\|u\|_2=1} &|\{x \in
 \{0,1\}^n: \sum_{i=1}^n x_i |(u^\T W)_i| \leq tG\}| \nonumber \\
 &\leq 60 (357 n^6)^{m+2} e^{2\sqrt{2nG}}/\delta = n^{O(m)} e^{2\sqrt{2nG}}/\delta. 
 \label{eq:bg-ub-2}
 \end{align}
 
By \Cref{thm:bnb} and the union bound, with probability at least
\[
1-\Pr_{A,c}[{\rm IPGAP}(A,b,c) \geq G]-\delta-e^{-n/5},
\] 
we have that the size of the branch-and-bound tree is at most 
\[
 n^{O(m)} e^{2\sqrt{2 G n}}/\delta. \quad \qed
\]
\end{proof}
 
 We now sketch the high level ideas of the proof of~\Cref{lem:knap-bnd}. 
 For $g \geq 0$, $u \in \mathbb{S}^{d-1}$, define the knapsack
 \[
 K(u,g) := \{x \in \{0,1\}^n: \sum_{i=1}^n x_i |u^\T W_i| \leq g\}.
 \]
 
 The proof of the lemma will proceed in two steps. In the first step, we
 show that the expected size of the knapsack polytope $K(u,G)$ satifies
 $\E_W[|K(u,G)|] \leq e^{2\sqrt{2nG}}$, for any fixed $u \in \mathbb{S}^{d-1}$
 and $G \geq 0$ (see~\Cref{lem:one-knapsack}). In the
 second step, we extend this bound to all $u \in \mathbb{S}^{d-1}$ via the
 union bound applied to a carefully constructed net of knapsacks of the form
 $K(u,G+1/n)$ for $u \in \mathcal{K}$, where $\mathcal{K} \subseteq \mathbb{S}^{d-1}$ will have
 size $|N| = n^{O(d)}$. The slight increase in capacity $G \rightarrow G+1/n$
 is to ensure that every $K(u,G)$ knapsack is contained in some $K(u',G+1/n)$
 knapsack, for some $u' \in \mathcal{K}$. To ensure this, we rely
 on~\Cref{lem:logcon-matrix} below to help us control the distance
 between knapsacks induced by nearby $u$'s.   
 
 One complicating factor in the construction of $\mathcal{K}$ is the lack of
 any bound on the norms of the column means $\mu_i := \E[W_i]$, $i \in [n]$.
 To deal with arbitrary means, we will make use of a hyperplane arrangement
 $\mathcal{H}$ induced by the $\mu_i$'s, such that for any full-dimensional
 cell $C$ of $\mathcal{H}$ and $i \in [n]$, we have that $|u^\T \mu_i|$ for $u
 \in C$ either lies in a small interval or is so large that $x_i = 0$ for any
 $x \in K(u,G)$.
 
 We now give our main bound on the expected size of knapsack polytopes with
 random weights.
 
 \begin{lemma}
 \label{lem:one-knapsack}
 Let $\omega_1,\dots,\omega_n \in \R$ be independent continuous random
 variables with maximum density at most $1$. Then, for any $g \geq 0$, we have
 \[
 \E[|\{x \in \{0,1\}^n: \sum_{i=1}^n x_i |\omega_i| \leq g\}|] \leq e^{2\sqrt{2 n g}}.
 \]
 \end{lemma}
 \begin{proof}
 Let $K := \{x \in \{0,1\}^n: \sum_{i=1}^n x_i |\omega_i| \leq g\}$. For any $\gamma \geq 0$, we first note that 
 \begin{equation}
 \label{eq:knp-exp-bnd}
 |K| \leq e^{\gamma g} \prod_{i=1}^n (1+e^{-\gamma \omega_i}).
 \end{equation}
 
 To see this, note that each $x \in \{0,1\}^n$ can be associated with the
 term $e^{\gamma (g-\sum_{i=1}^n x_i \omega_i)}$ on the right hand side (after
 expanding out the product) and that each term with $x \in K$ contributes at
 least $1$. 
 
 For each $i \in [n]$, letting $f_i: \R \rightarrow \R_+$ be the probability density of $\omega_i$, we have that
 \begin{equation}
 \E[e^{-\gamma |\omega_i|}] = \int_0^\infty e^{-\gamma x}(f_i(x)+f_i(-x)) dx 
 \leq 2 \int_0^\infty e^{-\gamma x} dx = \frac{2}{\gamma}, \label{eq:neg-exp}
 \end{equation}
 where we have used the assumption that $\max_{x \in \R} f_i(x) \leq 1$, $\forall i \in [n]$.
 
 Combining~\eqref{eq:knp-exp-bnd},~\eqref{eq:neg-exp}, using that
 $\omega_1,\dots,\omega_n$ are independent, we get that
 \[
 \E[|K|] \leq e^{\gamma g} \prod_{i=1}^n \E[1 + e^{-\gamma \omega_i}]
              \leq e^{\gamma g} (1 + \frac{2}{\gamma})^n
              \leq e^{\gamma g + \frac{2n}{\gamma}}. 
 \]
 Setting $\gamma = \sqrt{\frac{2n}{g}}$, we get that $\E[|K|] \leq
 e^{2\sqrt{2g n}}$, as claimed. \qed
 \end{proof}
We remark that an $e^{\Omega(\sqrt{nG})}$ dependence above is necessary.
This holds for $\omega_1,\dots,\omega_n$ uniform in $[0,1]$ and $G \leq n$.
Letting $S = \{i \in [n]: \omega_i \leq \sqrt{G/n}\}$, note that any subset of
at most $\floor{\sqrt{nG}}$ elements of $S$ fits inside the knapsack $K$. It is easy
to verify that $\E[|S|] = n(\sqrt{G/n}) = \sqrt{nG}$ and that $\Pr[|S| \geq
\floor{\sqrt{nG}}] \geq 1/2$. In particular, 
\[
\E[|K|] \geq \E[|\{T: T \subseteq S, |T| \leq \floor{\sqrt{nG}}\}|] \geq \frac{1}{2} 2^{\floor{\sqrt{Gn}}} = e^{\Omega(\sqrt{n G})}, \text{ as needed.}
\]
 
The next lemma give us control on the distance between knapsack weights
induced by nearby $u$'s. The proof follows along the same lines as the upper
bound in~\Cref{lem:gauss-matrix}, using~\Cref{thm:log-main} to give the
requisite tailbounds. 
 
 \begin{lemma}
 \label{lem:logcon-matrix}
Let $n \geq 100 d$ and $W := (W_1,\dots,W_n) \in \R^{d \times n}$ be a matrix
whose columns are independent logconcave random vectors with identity
covariance. Then, 
 \[
 \Pr[\max_{\|u\|_2=1} \|u^\T (W-\E[W])\|_1 \geq 4n] \leq e^{-n/5} . 
 \]   
 \end{lemma}
 \begin{proof}
 Since the statement of the lemma is invariant to adding a fixed matrix to
 $W$, we assume without loss of generality that $\E[W] = 0$. For $i \in [n]$,
 $\|u\|_2 = 1$, by \Cref{thm:log-marg} we see that $u^\T W_i$ is logconcave.
 Furthermore, $\Var[u^\T W_i] = \E[(u^\T(W_i-\E[W_i]))^2] = \|u\|_2^2 = 1$.
 Therefore, for $\lambda \in [0,1)$, by \Cref{thm:log-main} part 1 we have
 that:
 \begin{align*}
 \E[e^{\lambda |u^\T W_i|}] 
   &= \int_0^\infty \Pr[e^{\lambda |u^\T W_i|} \geq t] dt 
    = \int_0^\infty \Pr[|u^\T W_i| \geq \log t/\lambda] dt \\
    &\leq \int_0^\infty \min \{1, e^{1-\log t/\lambda}\} dt 
     = \int_0^\infty \min \{1, e t^{-1/\lambda} \} dt \\
    &= e^{\lambda} + e \int_{e^{\lambda}}^\infty t^{-1/\lambda} dt 
     = e^{\lambda} + e \left[ \frac{1}{1-1/\lambda} t^{1-1/\lambda} \right]_{e^{\lambda}}^\infty \\
    &= e^{\lambda} + \frac{\lambda}{1-\lambda} e^{\lambda} = \frac{e^{\lambda}}{1-\lambda}.  
 \end{align*}
 Therefore, for $\|u\|_2=1$ and $s \geq 2$, we have that
 \begin{align}
 \Pr[\sum_{i=1}^n |u^\T W_i| \geq s n] &\leq \min_{\lambda \in [0,1)} e^{-\lambda sn} \E[e^{\sum_{i=1}^n \lambda |u^\T W_i|}] \nonumber \\
       &\leq \min_{\lambda \in [0,1)} e^{-\lambda sn} \frac{e^{\lambda n}}{(1-\lambda)^n} = e^{-n(s-2-\log(s-1))}, \label{eq:exp-tail}
 \end{align}
 where the minimum is attained at $\lambda = \frac{s-2}{s-1} \in [0,1)$.
 Letting $N_\epsilon$ be a minimal $\epsilon$-net of $\mathbb{S}^{d-1}$, similar to the computation for~\eqref{eq:upper-plus}, we get that
 \begin{equation}
 \max_{\|u\|_2=1} \|u^\T W\|_1 \leq \frac{1}{1-\epsilon} \max_{u \in N_\epsilon} \|u^\T W\|_1.\label{eq:eps-net-comp}
 \end{equation}
 Using the above for $\epsilon = \frac{1}{4}$, we deduce the desired probability bound
 \begin{align*}
 \Pr[\max_{\|u\|_2=1} \|u^\T W\|_1 \geq 4n]
 &\underbrace{\leq}_{\text{by \eqref{eq:eps-net-comp}}} \Pr[\max_{u \in N_{1/4}} \|u^\T W\|_1 \geq 3n] \\
 &\underbrace{\leq}_{\text{by \eqref{eq:exp-tail}}} |N_{1/4}| e^{-n(1-\log(2))}
 \leq 9^d e^{-n/4} \underbrace{\leq}_{n \geq 100 d} e^{-n/5}.  \quad \qed
 \end{align*}
 \end{proof}
 
 We now have all the ingredients needed to prove~\Cref{lem:knap-bnd}.
 
 \begin{proof}[of \Cref{lem:knap-bnd}] 
 Let $K(u,g) := \{x \in \{0,1\}^n: \sum_{i=1}^n x_i |u^\T W_i| \leq g\}$ for
 $u \in \mathbb{S}^{d-1}$, $g \geq 0$. Noting that $|K(u,G)| \leq |\{0,1\}|^n
 \leq 2^n$, we may assume that $G \leq n$ since otherwise the bound of $n^{O(m)}
 e^{2\sqrt{2nG}}$ follows trivially. 
 
 We begin by constructing a suitable net of knapsacks as described at the
 beginning of the section. Let $\mu_i := \E[W_i]$, $i \in [n]$, and let
 $\mathcal{H}$ in denote the hyperplane arrangement on $\R^d$ induced by the
 hyperplanes $u^\T \mu_i = \frac{j}{2n^2}$, $i \in [n]$, $j \in
 \{-10n^3,\dots,10n^3\}$. Noting that this arrangement has $l := n(20n^3+1)$
 hyperplanes, it is well-known that the number of $d$-dimensional cells of
 $\mathcal{H}$ is at most $\sum_{i=0}^d \binom{l}{i} \leq l^{d+1} \leq
 (21n^4)^{d+1}$ (see for example~\cite[Proposition 6.1.1]{Matousek13}).
 Letting $\epsilon = 1/(8n^2)$, for each $d$-cell $C$ of $\mathcal{H}$, let
 $N^C_\epsilon$ denote a minimal $\epsilon$-net of $\mathbb{S}^{d-1} \cap C$.
 Finally, we let $\mathcal{K} := \cup_{C} N_\epsilon^C$, where $C$ ranges over
 all $d$-cells of $\mathcal{H}$. The size of $\mathcal{K}$ is bounded by
 \[
 |\mathcal{K}| \leq (1+2/\epsilon)^d (21n^4)^{d+1} \leq (1+16n^2)^d
 (21n^4)^{d+1} \leq (357 n^6)^{d+1}.
 \]
 
 \begin{claim} 
 Let $E_1$ denote the event that for all $u \in
 \mathbb{S}^{d-1}$, there exists $u' \in \mathcal{K}$ such that $K(u,G)
 \subseteq K(u',G+1/n)$. Then, $E_1$ holds with probability at least $1-
 e^{-n/5}$.
 \end{claim}
 \begin{proof}
 Let $E'_1$ denote the event that $\|u^\T(W-\E[W])\|_1 < 4n$, for all $\|u\|_2
 \in \mathbb{S}^{d-1}$. By~\Cref{lem:logcon-matrix}, we see that $E_1'$ holds
 with probability at least $1-e^{-n/5}$. To prove the claim, we condition
 on $E'_1$ and show that $E_1$ holds. 
 
 Take $u \in \mathbb{S}^{d-1}$. Let $C$ denote a $d$-cell of $\mathcal{H}$
 containing $u$, and let $u' \in \mathcal{K}$ denote the closest point in
 $N^C_\epsilon \subseteq C \cap \mathbb{S}^{d-1}$ to $u$. 
 
 Let $B \subseteq [n]$ denote the (possibly empty) subset of indices such that
 either $\mu_i^\T v \leq -5n$ or $\mu_i^\T v \geq 5n$ is valid for all $v \in C$. Let $A
 = [n] \setminus B$. Since $5n = \frac{10 n^3}{2n^2}$, for all $i \in A$,
 there exists $j_i \in \{-10n^3,\dots,10n^3-1\}$ such that $\frac{j_i}{2n^2}
 \leq \mu_i^\T v \leq \frac{j_i+1}{2n^2}$ is valid for all $v \in C$.
 In particular this implies that for all $i\in A$, $u,u' \in C$, we have \begin{equation}
 |(u-u')^\T \mu_i |\leq \frac{1}{2n^2}. \label{eq:cell-bnd}
 \end{equation}
 
 We first show that if $x \in K(u,G)$, then $x_i = 0$, $\forall i \in B$. For
 the sake of contradiction, assume $x \in K(u,G)$ and $x_i = 1$ for some $i \in B$.
 Then, since $G \leq n$, we have that
 \begin{align*}
 \sum_{j=1}^n x_j|u^\T W_j| &\geq |u^\T W_i| \geq |u^\T \mu_i| - |u^\T(W_i-\mu_i)| \\ &\underbrace{\geq}_{i \in B} 5n - \|u^\T(W-\E[W])\|_1 \underbrace{>}_{\text{ by } E_1'} 5n - 4n \geq G,   
 \end{align*}
 a clear contradiction to the assumption that $x \in K(u,G)$. 
 
 Take $x \in K(u,G)$. We now show that $x \in K(u',G+1/n)$ as follows, 
 \begin{align*}
 G &\geq \sum_{i=1}^n x_i |u^\T W_i| \geq \sum_{i=1}^n x_i(|(u')^\T W_i| - |(u-u')^\T (W-\mu_i)| - |(u-u')^\T \mu_i|) \\
   &\geq \left(\sum_{i=1}^n x_i|(u')^\T W_i|\right) - \|(u-u')^\T(W-\E[W])\|_1 - \sum_{i \in A} 1/(2n^2) \\ & \quad \left(\text{ since } x_i = 0, \forall i \in B \text{ and } \eqref{eq:cell-bnd} \right) \\
   &\geq \left(\sum_{i=1}^n x_i|(u')^\T W_i|\right) - 4n\epsilon - |A|/(2n^2) \geq \left(\sum_{i=1}^n x_i|(u')^\T W_i|\right) - 1/n. \quad \qed 
 \end{align*} 
 \end{proof} 
 
 For $u \in \mathbb{S}^{d-1}$, by \Cref{thm:log-marg} we know that $u^\T W_i$, for $i \in [n]$,
 are independent and logconcave. Furthermore, $\Var[u^\T W_i] =
 \E[u^\T(W_i-\mu_i)^2] = \|u\|_2^2=1$, $\forall i \in [n]$. Therefore,
 by~\Cref{thm:log-main} part 2, the densities of $u^\T W_i$, $i \in [n]$, have
 maximum density at most $1$. Applying~\Cref{lem:one-knapsack} with $\omega_i
 = u^\T W_i$, $i \in [n]$ and $g = G+1/n$, we get that $\E_W[|K(u,G+1/n)|]
 \leq e^{2\sqrt{2n(G+1/n)}} \leq e^{2\sqrt{2nG} + 4}$.
 
 Let $E_2$ denote the event $\forall u \in \mathcal{K}$, $|K(u,G+1/n)| \leq
 |\mathcal{K}| e^{2\sqrt{2nG}+4}/\delta$ for $\delta \in (0,1)$. By Markov's
 inequality, for $u \in \mathcal{K}$ we have that 
 \[
 \Pr[|K(u,G+1/n)| \geq |\mathcal{K}| e^{2\sqrt{2nG}+4}/\delta] \leq \delta/(|\mathcal{K}|).
 \]
 Therefore, by the union bound, $E_2$ occurs with probability at least
 $1-\delta$.
 
 By the above claim, noting that
 \[
 |\mathcal{K}| e^{2\sqrt{2nG}+4} \leq (357 n^6)^{d+1} e^{2\sqrt{2nG}+4}
 \leq 60(357 n^6)^{d+1} e^{2\sqrt{2nG}}, 
 \]
 we see that 
 \begin{align*}
 \Pr[\max_{u \in \mathbb{S}^{d-1}} |K(u,G)| \leq 60(357 n^6)^{d+1} e^{2\sqrt{2nG}}/\delta] &\geq 1-\Pr[\neg E_1] - \Pr[\neg E_2]\\& \geq 1-\delta-e^{-n/5}, 
 \end{align*}
 as needed.  \qed
 \end{proof}

\bibliographystyle{splncs04}
\bibliography{bibliography}

\begin{thebibliography}{10}
\providecommand{\url}[1]{\texttt{#1}}
\providecommand{\urlprefix}{URL }
\providecommand{\doi}[1]{https://doi.org/#1}

\bibitem{beier_core_2004}
Beier, R., V{\"{o}}cking, B.: Probabilistic analysis of knapsack core
  algorithms. In: Munro, J.I. (ed.) Proceedings of the Fifteenth Annual
  {ACM-SIAM} Symposium on Discrete Algorithms, {SODA} 2004, New Orleans,
  Louisiana, USA, January 11-14, 2004. pp. 468--477. {SIAM} (2004)

\bibitem{dey_branch-and-bound_2021}
Dey, S.S., Dubey, Y., Molinaro, M.: Branch-and-bound solves random binary {IPs}
  in polytime. In: Proceedings of the 2021 ACM-SIAM Symposium on Discrete
  Algorithms (SODA), pp. 579--591. Society for Industrial and Applied
  Mathematics (Jan 2021). \doi{10.1137/1.9781611976465.35}

\bibitem{doerr}
Doerr, B.: Analyzing Randomized Search Heuristics: Tools from Probability
  Theory, pp. 1--20. World Scientific (2011). \doi{10.1142/9789814282673\_0001}

\bibitem{dyer_gap_1992}
Dyer, M., Frieze, A.: Probabilistic analysis of the generalised assignment
  problem. Math. Program.  \textbf{55}(1-3),  169--181 (Apr 1992).
  \doi{10.1007/bf01581197}

\bibitem{dyer_probabilistic_1989}
Dyer, M., Frieze, A.: Probabilistic analysis of the multidimensional knapsack
  problem. Mathematics of OR  \textbf{14}(1),  162--176 (Feb 1989).
  \doi{10.1287/moor.14.1.162}

\bibitem{fritz2020}
Eisenbrand, F., Weismantel, R.: Proximity results and faster algorithms for
  integer programming using the steinitz lemma. ACM Trans. Algorithms
  \textbf{16}(1),  1--14 (Jan 2020). \doi{10.1145/3340322}

\bibitem{vershynin}
Eldar, Y.C., Kutyniok, G. (eds.): Compressed Sensing. Cambridge University
  Press (2009). \doi{10.1017/cbo9780511794308}

\bibitem{Feller08}
Feller, W.: An introduction to probability theory and its applications, vol 2.
  John Wiley \& Sons (1991)

\bibitem{Fradelizi99}
Fradelizi, M.: Hyperplane sections of convex bodies in isotropic position.
  Beitr{\"a}ge Algebra Geom  \textbf{40}(1),  163--183 (1999)

\bibitem{FK89}
Furst, M.L., Kannan, R.: Succinct certificates for almost all subset sum
  problems. SIAM Journal on Computing  \textbf{18}(3),  550--558 (1989)

\bibitem{galvin_three_2014}
Galvin, D.: Three tutorial lectures on entropy and counting. arXiv:1406.7872
  [math]  (Jun 2014)

\bibitem{goldberg_finding_1984}
Goldberg, A., Marchetti-Spaccamela, A.: On finding the exact solution of a
  zero-one knapsack problem. In: Proceedings of the sixteenth annual ACM
  symposium on Theory of computing - STOC '84. ACM Press (1984).
  \doi{10.1145/800057.808701}

\bibitem{jansen2019}
Jansen, K., Rohwedder, L.: Integer programming (Oct 2019).
  \doi{10.1002/9781119454816.ch10}

\bibitem{kannan1987}
Kannan, R.: Minkowski's convex body theorem and integer programming.
  Mathematics of OR  \textbf{12}(3),  415--440 (Aug 1987).
  \doi{10.1287/moor.12.3.415}

\bibitem{lenstra1983}
Lenstra, H.: Integer programming with a fixed number of variables. Mathematics
  of OR  \textbf{8}(4),  538--548 (Nov 1983). \doi{10.1287/moor.8.4.538}

\bibitem{LV07}
Lov{\'a}sz, L., Vempala, S.: The geometry of logconcave functions and sampling
  algorithms. Random Structures \& Algorithms  \textbf{30}(3),  307--358 (2007)

\bibitem{lueker_average_1982}
Lueker, G.S.: On the average difference between the solutions to linear and
  integer knapsack problems. In: Applied Probability-Computer Science: The
  Interface Volume 1, pp. 489--504. Birkh\"auser Boston (1982).
  \doi{10.1007/978-1-4612-5791-2}

\bibitem{Matousek13}
Matousek, J.: Lectures on discrete geometry, vol.~212. Springer Science \&
  Business Media (2013)

\bibitem{papadimitriou1981}
Papadimitriou, C.H.: On the complexity of integer programming. J. ACM
  \textbf{28}(4),  765--768 (Oct 1981). \doi{10.1145/322276.322287}

\bibitem{pataki_basis_2010}
Pataki, G., Tural, M., Wong, E.B.: Basis {{Reduction}} and the {{Complexity}}
  of {{Branch}}-and-{{Bound}}. In: Proceedings of the {{Twenty}}-{{First Annual
  ACM}}-{{SIAM Symposium}} on {{Discrete Algorithms}}. pp. 1254--1261. {Society
  for Industrial and Applied Mathematics} (Jan 2010).
  \doi{10.1137/1.9781611973075.100}

\bibitem{Prekopa71}
Pr{\'e}kopa, A.: Logarithmic concave measures with application to stochastic
  programming. Acta Scientiarum Mathematicarum  \textbf{32},  301--316 (1971)

\bibitem{Rglin2007}
R\"oglin, H., V\"ocking, B.: Smoothed analysis of integer programming. Math.
  Program.  \textbf{110}(1),  21--56 (Jan 2007).
  \doi{10.1007/s10107-006-0055-7}

\end{thebibliography}

\end{document}